\documentclass[a4paper,11pt]{amsart}
\usepackage[margin=3cm]{geometry}

\usepackage[T1]{fontenc}
\usepackage{lmodern, amsfonts,amsmath,amstext,amsbsy,amssymb,
amsopn,amsthm,upref,eucal,mathptmx,mathtools,url}
\usepackage{cleveref}

\usepackage{mathrsfs}

\RequirePackage{xcolor} 
\definecolor{halfgray}
{gray}{0.55}
\definecolor{webgreen}
{rgb}{0,0.4,0}
\definecolor{webbrown}
{rgb}{.8,0.1,0.1}
\definecolor{red}
{rgb}{1,0,0}
\usepackage{microtype}
\usepackage{tikz}

\newcommand \R {{ \mathbb R}}
\def\C{{\mathbb C}}

\newcommand \Z {{ \mathbb Z}}
\newcommand \N {{ \mathbb N}}

    \newcommand{\qandq}{\quad \text{and} \quad}

    \newcommand{\cL}{\mathcal{L}}

    \newcommand{\subof}{\subset}
    \newcommand{\ti}{\times}
    \newcommand{\sans}{\setminus}

    \newcommand{\Lam}{\Lambda}
    
    \newcommand{\Gam}{\Gamma}
    
    \newcommand{\Sig}{\Sigma}
    \newcommand{\sig}{\sigma}

    \newcommand{\al}{\alpha}

    \newcommand{\om}{\omega}
    \newcommand{\Om}{\Omega}
    
    \newcommand{\bbR}{\mathbb{R}}
    \newcommand{\bbC}{\mathbb{C}}

    \newcommand{\bbN}{\mathbb{N}}

    \newcommand{\invm}{^{-m}}
    \newcommand{\invn}{^{-n}}
    \newcommand{\dth}{d_\theta}

\newcommand{\infnorm}{H}

\newcommand*{\diff}{\mathop{}\!\mathrm{d}}

\newcommand{\norm}[1]{\left\lVert#1\right\rVert}

\newcommand{\infpq}[1]{\|#1\|_{\infty, p, q}}

\newcommand{\infpqq}[1]{\|#1\|_{\infty, p, q+1}}

\newcommand{\ltpq}[1]{|#1|_{\theta, p, q} }

\newcommand{\lttpq}[1]{|#1|_{\sqrt{\theta}, p, q} }
\newcommand{\ltttpq}[1]{|#1|_{\beta, p, q} }

\newcommand{\Lpq}[1]{\|#1\|_{\theta, p, q} }

\newcommand{\infldz}[1]{\|#1\|_{\infty, \ell +2, 0}}

\newcommand{\ltldz}[1]{|#1|_{\theta, \ell +2, 0} }

    \newcommand{\pq}{{p,q}}
    \newcommand{\pqq}{{p,q+1}}
    \newcommand{\ipq}{{\infty,p,q}}
    
    \newcommand{\ipqq}{{\infty,p,q+1}}
    \newcommand{\tpq}{{\theta,p,q}}
    
    \newcommand{\tpqq}{{\theta,p,q+1}}

\DeclareMathOperator{\Leb}{Leb}

\DeclareMathOperator{\distance}{dist}

\newtheorem{bigthm}{Theorem} 

\newtheorem{theorem}{Theorem}[section]
\newtheorem {lemma}[theorem]{Lemma}
\newtheorem {proposition}[theorem]{Proposition}
\newtheorem{corollary}[theorem]{Corollary}

\newtheorem{definition}[theorem]{Definition}

\title{Quantitative infinite mixing for non-compact skew products}
 
\author{Paolo Giulietti}
\address{Dipartimento di Matematica, Universit\'{a} di Pisa \\ 
Largo Bruno Pontecorvo 5, 56127 Pisa, Italy}
\email{paolo.giulietti@unipi.it}

\author{Andy Hammerlindl}
\address{Monash University, School of Mathematics \\
  Victoria 3800, Australia}
\email{andy.hammerlindl@monash.edu}

\author{Davide Ravotti}
\address{Universit\"{a}t Wien, Department of Mathematics \\
  Oskar-Morgenstern-Platz 1, 1090 Wien, Austria}
\email{davide.ravotti@gmail.com}

\begin{document}

\maketitle

\begin{abstract}
We consider skew products over subshifts of finite type in which the fibers are copies of the real line, and we study their mixing properties with respect to any infinite invariant measure given by the product of a Gibbs measure on the base and Lebesgue measure on the fibers. Assuming that the system is accessible, we prove a quantitative version of Krickeberg mixing for a class of observables which is dense in the space of continuous functions vanishing at infinity.   
\end{abstract}

\section{Introduction}

Mixing of dynamical systems is a well-studied and well-understood concept in the
finite measure setting. Suppose a dynamical system $F : Y \to Y$ preserves a
probability measure $\nu$. This system is mixing if for any two measurable sets $A$ and
$B$, the limit $\lim_{n\to\infty} \nu(F^{-n}(A) \cap B)$ exists and is equal to
$\nu(A) \nu(B)$.
This property can be formulated equivalently in terms of integrals, where
\[
    \lim_{n\to\infty} \int (r \circ F^n) \cdot s \, \diff \nu
    = \int r \, \diff \nu \int s \, \diff \nu
\]
holds for any measurable $r$ and $s$.
The study of mixing is considerably more complicated when the invariant
measure $\nu$ is infinite. As a matter of fact, extending the definition of
mixing to the infinite measure case has been an important issue for many years
in infinite ergodic theory.


Different definitions that describe \lq\lq mixing behaviours\rq\rq\ have been proposed and studied; notably \emph{Krickeberg mixing} and \emph{global-local mixing}. 
A measure preserving system $(Y, \nu, F)$, where $\nu$ is a $\sigma$-finite, infinite measure, is said to be Krickeberg mixing if there exists a sequence of positive numbers (a \emph{scaling rate}) which is asymptotic to $\int_Y r \circ F^n \cdot \overline{s} \diff \nu$, for any pair of continuous complex valued functions $r,s$ with compact support, see \S\ref{Sec:Km} for a precise definition.
Note that such a definition is akin to saying that the system satisfies a local limit theorem.
On the other hand, global-local mixing for $(Y, \nu, F)$ means that, for an integrable function $s \in L^1(Y,\nu)$ and for an essentially bounded function $r \in L^{\infty}(Y,\nu)$ which, roughly speaking, has an \lq\lq infinite volume limit\rq\rq, the sequence $\int_Y  r \circ F^n \cdot \overline{s} \diff \nu$ converges to a (possibly non-zero) constant. Both definitions of mixing are relevant, as they capture different properties of physical systems, see \cite{DoNa} for a more detailed discussion of such matters.

 In fact, Krickeberg mixing and global-local mixing properties of infinite measure preserving systems have been studied by many authors,  both for abstract dynamical systems as well as physically modeled systems. Relevant papers include \cite{Aar, DoNa2,FlRa,Gou,Mel,MeTe1,MeTe2,OhPa,Pen,DoNaPe} for Krickeberg mixing, and \cite{BLG,DLN,DoNa,Len1,Len2} for global-local mixing.

In this paper, we focus on the case of \emph{skew products} which have $\R$ as fibers, namely we consider product spaces $Y = X \times \R$, where $X$ is a compact metric space, and maps of the form $F(x,r) = (Tx, r+f(x))$, where $(X, \mu, T)$ is a probability preserving system and $f\colon X \to \R$ is a sufficiently regular function. Under some assumptions on the base dynamics, we prove a quantitative version of Krickeberg mixing for smooth observables with sufficiently fast decay at infinity.

In a previous work \cite{GHR}, we proved quantitative global-local mixing results for the same class of systems. The key tool to inverstigate mixing properties is the \emph{transfer operator}; more precisely, after taking the Fourier transform on the fibers of the skew product, we are led to study a family of twisted transfer operators, parametrized by frequency. For large frequencies, an accessibility assumption of the system guarantees that the norms of the transfer operators decay rapidly (i.e., faster than any polynomial). For frequencies close to zero, the main contribution comes from an analytic perturbation of the leading isolated eigenvalue of the \lq\lq untwisted\rq\rq\ transfer operator. 

We stress that both global-local mixing and Krickeberg mixing stem from integrating the contributions of the action of the twisted transfer operators on the observables near the zero frequency. In \cite{GHR}, we chose as global observables functions which are fiberwise positive definite (i.e., the Fourier-Stieltjes transform of measures), so that the speed of convergence is dictated by the behaviour near zero of the corresponding measure. 
Here, we are interested in Krickeberg mixing, which can be seen as a form of mixing where both observables are local. 
In particular, in this work, a careful choice of the space of observables ensures that the dependency on the frequency parameter is smooth, so that we can carry out a detailed analysis of the integral near zero and prove precise asymptotics. 

\subsection{Organization  of the paper} 
The paper is organised as follows. In Section \ref{Sec:2} we define the class of systems we consider and set the notation, introduce the accessibility condition and the class of observables, and we state our main results.
Section \ref{Sec:3} contains a preliminary and technical result which describes the asymptotics of certain integrals. 
In Section \ref{Sec:4}, we introduce the twisted transfer operators and we study the integrals of these operators applied to families of functions depending smoothly on the parameter.
In Section \ref{Sec:z}, we study the smoothness of the Fourier transform along the vertical fibers of our observables.
In Section \ref{sec:rapid_decay}, we recall some of the results in \cite{GHR} which are needed in the present work.
Section \ref{Sec:5} is devoted to the proof of the main result in the case of skew products over a one-sided shift.
In Sections \ref{Sec:6} and \ref{sec:redu}, we describe how it is possible to reduce the problem from a skew product over a two-sided shift to the setting of Section \ref{Sec:5}. Finally, in Section \ref{Sec:7}, we complete the proof of the main result for skew products over a two-sided shift.

\section{Statement of the main results}\label{Sec:2}

We will primarily focus on skew products with real fibers over a subshift  of finite type, both in the case of one-sided and two-sided shift. We will detail the class of observables we are interested in each case. Note that by
way of symbolic dynamics (see, e.g., \cite{Bow}),
the results here can be adapted 
to deal with skew products
$F(x,t) = (A(x), t + f(t))$ defined
over a diffeomorphism $A : M \to M$ where the invariant measure on
the smooth manifold $M$
is a Gibbs measure supported on a transitive uniformly hyperbolic subset
$\Lam\subset M$.
The relevant measures pull back from the smooth setting to the symbolic setting
without any problems.
Moreover, the necessary regularity of the local observables is preserved by the symbolic conjugacy.
However, 
one has to be careful in controlling how the accessibility of the original smooth system  defined on $M \times \bbR$ translates into the accessibility in the symbolic setting.
This issue is studied in detail in Appendix A of \cite{GHR}
and so we do not discuss it here further.
For simplicity, we state all of the results below in the symbolic setting.

\subsection{Skew products and their accessibilities properties}

Let $\sigma \colon \Sigma \to \Sigma$ be a topologically mixing two-sided subshift of finite type. 
For  $0<\theta < 1$, let $d_\theta$ be the distance on $\Sigma$ defined by
\[
d_\theta(x,y) = \theta^{\max\{j \in \N\ :\ x_i = y_i \text{\ for all\ } |i|<j\}}.
\]
We fix a H\"{o}lder potential $u \colon \Sigma \to \R$, and we indicate by $\mu$ the associated Gibbs measure on $\Sigma$.

Let us denote by $\mathscr{F}_{\theta}$ the space of Lipschitz continuous functions $w:\Sigma \to \mathbb{C}$, equipped with the norm
\[
\norm{w}_{\theta} = \norm{w}_{\infty} + |w|_{\theta}, \text{\ \ \ where\ \ \ } |w|_{\theta} = \sup_{x\neq y} \frac{|w(x)-w(y)|}{d_{\theta}(x,y)}.
\]
Up to replacing $\theta$ by a larger value, we can assume that $u$ is Lipschitz. 
Henceforth, $\theta$ is fixed so that the potential $u$ is Lipschitz.

Let $f \in \mathscr{F}_{\theta}$ be a real-valued Lipschitz function with zero average, namely such that $\int_\Sigma f \diff \mu = 0$.
We consider the skew product $F \colon \Sigma \times \R \to \Sigma \times \R$ defined by 
\begin{equation}\label{eq:skew_product}
F\colon \Sigma \times \R \to \Sigma \times \R, \qquad F(x, t) = (\sigma x, t + f(x)).
\end{equation}
The map $F$ preserves the infinite measure $\nu = \mu \times \Leb$ on $\Sigma \times \R$, where $\Leb$ is the Lebesgue measure on $\R$.

The definitions above can be adapted to the case of a topologically mixing \emph{one-sided} subshift of finite type $\sigma^{+} \colon X \to X$. 
Given a real-valued Lipschitz function $f \colon X \to \R$ with zero average, define the associated (non-invertible) skew product $F^{+} \colon X \times \R \to X \times \R$ as $F^{+}(x,t) = (\sigma^{+}x, t+f(x))$.
The canonical projection $\pi \colon \Sigma \to X$ given by $\pi \colon x=(x_i)_{i\in \Z} \mapsto (x_i)_{i\geq 0}$ realizes $F^{+}$ as a factor of $F(x,t) = (\sigma x, f\circ \pi(x))$.

Two related notions that will play a key role in this work are \emph{accessibility} (for the invertible skew product $F$) and \emph{collapsed accessibility}  (for the non-invertible skew product $F^{+}$), which we now define.
For a point $x \in \Sigma$, we let
\[
\begin{split}
&W^s(x) = \{ y \in \Sigma : \text{there exists $n \in \Z$ such that $x_i = y_i$ for all $i\geq n$}\} \qquad \text{\ and }\\
&W^u(x) = \{ y \in \Sigma : \text{there exists $n \in \Z$ such that $x_i = y_i$ for all $i\leq n$}\} 
\end{split}
\]
be the stable and unstable sets at $x$ respectively. Let also $S_nf$ denote the Birkhoff sum of $f$ at $x$ up to time $n \in \N$, namely,
$S_nf(x) := \sum_{i=0}^{n-1}f\circ \sigma^i(x)$.
The \emph{stable set $W^s(x,t)$ at $(x,t) \in \Sigma \times \R$} is defined as 
\[
W^s(x,t) = \{ (y,r) \in \Sigma \times \R : y \in W^s(x) \text{ and } r-t = \lim_{n \to \infty} S_nf(x) - S_nf(y)\}.
\]
Equip $\Sigma \times \R$ with the metric $\distance((x,t),(y,r)) = d_{\theta}(x,y) + | r - t |$. 
Note that, for any $(y,r) \in W^s(x,t)$, the iterates $F^n(y,r) = (\sigma^ny, r+ S_nf(y))$ converge exponentially fast to $F^n(x,t)= (\sigma^nx, r+ S_nf(x))$ as $n \to \infty$ (with respect to the natural product metric on $\Sigma \times \R$). Similarly, the \emph{unstable set $W^u(x,t)$ at $(x,t) \in \Sigma \times \R$} is defined as 
\[
W^u(x,t) = \{ (y,r) \in \Sigma \times \R : y \in W^s(x) \text{ and } r-t = \lim_{n \to \infty} S_nf(\sigma^{-n}x) - S_nf(\sigma^{-n}y)\},
\]
and, for any $(y,r) \in W^u(x,t)$, the distance between $F^n(y,r)$ and $F^n(x,t)$ converges to zero exponentially fast as $n \to -\infty$. 

In what follows, we use superscripts $x^1,x^2,\ldots,x^N$ in order to distinguish a
finite sequence of points in $\Sigma$ from the symbols $x_i$ appearing in a
single element of the subshift.
An \emph{su-path} from $(x,t)$ to $(y,r)$ is a finite sequence $(x^i,t_i) \in \Sigma \times \R$, with $0\leq i \leq N$ for some $N \in \N$ such that $(x^0,t_0) = (x, t)$, $(x^N,t_N) = (y,r)$, and $(x^i,t_i)$ belongs to $W^s(x^{i-1},t_{i-1})$ or $W^u(x^{i-1},t_{i-1})$ for any $1\leq i\leq N$.
The invertible skew product $F$ is \emph{accessible} if there is a su-path from any point $(x,t)$ to any other point $(y,r)$.

For skew products $F^{+} \colon (x,t) \mapsto (\sigma^{+} x, t+f(x))$ over a one-sided shift $\sigma^{+} \colon X \to X$, we introduce the related notion of collapsed accessibility. 
We say that the Lipschitz function $f \colon X \to \R$ has the \emph{collapsed accessibility property} if there exists a constant $C \geq 0$ and $N \in \N$ such that for any $x \in X$, any $t\in [0,1]$, and $n \geq 2N$, there exist two sequences $x^0, x^1, \dots, x^{m+1} \in X$ and $y^0, \dots, y^{m} \in X$ for some $m \leq N$, such that
\begin{enumerate}
\item $x^0 = x^{m+1} = x$,
\item $(\sigma^{+})^n x^i = (\sigma^{+})^n y^i$ for all $i=0,\dots, m$,
\item $d_{\theta}(y^i, x^{i+1})\leq C\theta^n$ for all $i=0,\dots, m$,
\item $t = \sum_{i=0}^m S_nf(x^i)-S_nf(y^i)$.
\end{enumerate}
The following result, which was proved in \cite[Proposition 9.5]{GHR}, explains the relation between the accessibility and the collapsed accessibility property.
\begin{proposition}
    Let $f \colon X \to \R$ be a real-valued Lipschitz function. If the skew product $F\colon \Sigma \times \R \to \Sigma \times \R$ given by $F(x,t) = (\sigma x, t+f \circ \pi (x))$ is accessible, then $f$ has the collapsed accessibility property.
\end{proposition}

\subsection{Krickeberg mixing}\label{Sec:Km}

We are interested in studying the mixing properties of the skew products described above assuming the accessibility property (or the collapsed accessibility property for skew products over one-sided shifts). 
Proving Krickeberg mixing for such a skew product $F$ means finding a \emph{scaling rate}, namely a sequence $(\rho_n)_{n \in \N}$ of positive numbers $\rho_n$ diverging to infinity such that, for all pairs of bounded subsets $A,B \subset \Sigma \times \R$ (or $A,B \subset X \times \R$ respectively) whose boundaries have zero measure, the rescaled correlations converge to a non-zero limit, more precisely
\[
\lim_{n \to \infty} \rho_n \cdot \nu(F^{-n}A \cap B) = \nu(A) \, \nu(B).
\]
Equivalently, for all continuous, compactly supported functions $f,g \colon \Sigma \times \R \to \C$, 
\begin{equation}\label{def_K_mix}
\lim_{n \to \infty} \rho_n \cdot \int_{\Sigma \times \R} (f \circ F^n) \cdot \overline{g} \diff \nu = \nu(f) \nu(\overline{g}),
\end{equation}
where $\nu(f)$ denotes the integral of $f$ with respect to $\nu$ (and similarly for the one-sided case).

In this paper we are going to study the correlations for a special class of observables, described in the next subsection, which is dense in the space of continuous functions vanishing at infinity. 
In our main result, \Cref{cor:main_results} below, we find the scaling rate to be $\rho_n = 2 \sqrt{\pi \omega n}$ for a constant $\omega >0$  depending only on the skew-product $F$ and we prove an upper bound on the speed of convergence of \eqref{def_K_mix}.

\subsection{The observables}

Let $\mathscr{S}$ denote the space of Schwartz functions $w \colon \R \to \C$.
Given $w \in \mathscr{S}$, for all integers $p,q \geq 0$ we define
\[
\|w\|_{p,q}:= \sup \big\{ |t|^j \cdot \big\lvert w^{(\ell)}(t)\big\rvert : 0\leq j \leq p,\ 0\leq \ell\leq q,\ t \in \R \big\},
\]
where $w^{(\ell)}$ is the $\ell$-th derivative of $w$. We can rewrite the norm above as
\[
\|w\|_{p,q} = \sup \big\{ \max(1, |t|^p) \cdot \big\lvert w^{(\ell)}(t)\big\rvert :  0\leq \ell\leq q,\ t \in \R \big\},
\]
and, by construction, if $p\leq P$ and $q\leq Q$, then $\|w\|_{p,q} \leq \|w\|_{P,Q}$.

Our class $\mathscr{L}$ of observables consists of functions which are Schwartz on each \lq\lq vertical\rq\rq\ fiber of the skew product and whose dependence on the fiber is Lipschitz with respect to all the $\|\cdot\|_{p,q}$ norms; more precisely, we define
\[
\mathscr{L} = \{ s \colon \Sigma \to \mathscr{S} : s \text{ is Lipschitz with respect to $\|\cdot\|_{p,q}$ for all }p,q \in \Z_{\geq 0} \},
\]
and similarly for the one-sided case
\[
\mathscr{L}^{+} = \{ s \colon X \to \mathscr{S} : s \text{ is Lipschitz with respect to $\|\cdot\|_{p,q}$ for all }p,q \in \Z_{\geq 0} \}.
\]
Let $s \in \mathscr{L}$. For any $p,q \in \Z_{\geq 0}$, we define the sup norm
\[
\infpq{s} := \sup \{ \|s(x)\|_{p,q} :  x \in \Sigma\},
\]
the Lipschitz semi-norm
\[
\ltpq{s} := \sup \Bigg\{ \frac{\|s(x)-s(y)\|_{p,q}}{d_\theta(x,y)} : x,y\in \Sigma, x \neq y\Bigg\},
\]
and the Lipschitz norm
\[
\Lpq{s} := \infpq{s} +\ltpq{s}.
\]
The definitions for elements of  $\mathscr{L}^{+}$ are similar.

\subsection{Asymptotics of correlations}

Our main result in the invertible case is a quantitative version of Krickeberg mixing, i.e., of the limit given in \eqref{def_K_mix}, for observables in $\mathscr{L}$ under the accessibility assumption. 

\begin{bigthm}\label{cor:main_results}
Assume that the skew product $F \colon \Sigma \times \R \to \Sigma \times \R$ is accessible. 
There exists a constant $B>0$ such that, for every $r,s \in \mathscr{L}$ and for every $n \in \N$, we have
\[
\Bigg\lvert n^{1/2} \Bigg( \int_{X \times \R} (r \circ F^n) \cdot \overline{s}  \, \diff \nu \Bigg) -  \frac{\nu(r) \nu(\overline{s}) }{2\sqrt{\pi \omega}}  \Bigg\rvert \leq B \, \|r\|_{\theta, B, B} \,  \|s\|_{\theta, B, B} \, (\log n)^{10} n^{-1},
\]
where the constant $\omega >0$, defined in \S\ref{sec:tto}, depends on $F$ only.
\end{bigthm}

For the non-invertible case, assuming the collapsed accessibility property, we are able to obtain a much sharper result, namely a full asymptotic expansion for the correlations of observables in $\mathscr{L}^{+}$. 

\begin{bigthm}\label{thm:main_1_sd}
Assume that the skew product $F^{+} \colon X \times \R \to X \times \R$ has the collapsed accessibility property. There exist $B>0$ and a sequence of constants $(C_j)_{j \geq 0}$ such that the following holds. Let $r,s \in \mathscr{L}^+$ and let $k \in \Z_{>0}$.  There exists a constant $C(k,r,s)>0$ and, for every $ j \in \{1,3,\dots, 2k-1\}$, there exists $c_j(r,s)$, with 
\[
|c_j(r,s) | \leq C_j \, \|s\|_{\theta, j+2, 0} \, \|r\|_{\theta, j+2, 0}, \quad \text{and} \quad C(k,r,s) \leq C_k \, \|r\|_{\theta, Bk,Bk} \, \|s\|_{\theta, Bk,Bk},
\]
such that for all $n \in \N$ we have
\[
\Bigg\lvert \int_{X \times \R} (r \circ F^n) \cdot \overline{s}  \, \diff \nu -   \sum_{\substack{1\leq j\leq 2 k -1 \\ j \in 2\N +1}} c_j(r,s)  n^{-j/2} \Bigg\rvert \leq  C(k,r,s) \, n^{-k},
\]
where 
\[
c_1 = \frac{1}{2\sqrt{\pi \omega}} \nu(r) \overline{\nu(s)},
\]
for a constant $\omega >0$, defined in \S\ref{sec:tto}, which depends on $F^{+}$ only.
\end{bigthm}

\noindent
    \textbf{Remark.}
Several results in the literature share some similarities with our main results above, see, e.g., \cite{DoNaPe, Pe17, FePe}. 
The result closest to our work is \cite[Theorem 5.3]{FePe}, where the base dynamics is a subshift of finite type, hence it is in the same setting as our Theorem \ref{cor:main_results}, but with two key differences. 
One is that the authors assume a \emph{strong non-integrability} condition which (as the name suggests) is stronger than just the assumption of accessibility. 
The second difference is that they use observables which are functions either of the form $\Sigma \to \C$ or $\R \to \C$, 
whereas Theorem \ref{cor:main_results} considers, more generally, observables of the form $\Sigma \times \R \to \C$.

\section{A technical result}\label{Sec:3}

We will need to estimate the asymptotics of integrals of the form $\int_I [g(t)]^n v(t) \diff t$ for smooth functions $g,v \colon I \to \C$. When $g$ is the exponential of a real or of a purely imaginary function, the integral is controlled by the behaviour of the integrand around the stationary points and classical \lq\lq stationary phase\rq\rq-type results provide a power series expansion. 
Standard results of this type cannot be applied directly to our case, hence we provide a complete proof of the asymptotic expansion that we will use.

\begin{theorem}\label{thm:exp_int}
    Let $I \subof \bbR$ be a compact interval centered at zero and
    let $v \colon I \to \bbC$ and $g \colon I \to \bbC$ be smooth functions such that
    $g(0) = 1, g''(0) \in \R_{< 0},$ and $|g(x)| < 1$ for all $x \ne 0.$
    Then for any $k \in\bbN$, there are constants
    $c_1, c_3, c_5, \ldots, c_{2k-1} \in \bbC$ and $R > 0$
    such that
    \[
        \left|
        c_1 n^{-1/2} + c_3 n^{-3/2} + c_5 n^{-5/2} + \cdots + c_{2k-1} n^{-k+1/2}
        \ - \ 
        \int_I [g(t)]^n v(t) \, \diff t
        \right|
        \ < \ R n^{-k}
    \]
    for all $n.$
    Moreover,
    \begin{enumerate}
        \item the constant $c_j$ only depends on the derivatives
        \[
            g(0), g'(0), g''(0), \ldots, g^{(j+1)}(0)
            \qandq
            v(0), v'(0), v''(0), \ldots, v^{(j)}(0);
        \]
        \item
        the constant $R$ only depends
        on the derivatives of $g$ and $v$ up to order $2k + 6;$ and
        \item
        the constant $c_1$ is given by
        \begin{math}
            \begin{displaystyle}
                c_1 = \sqrt{\frac{\pi}{\om}} v(0) \end{displaystyle} \end{math}
        where $g''(0) = -2 \om.$
    \end{enumerate} \end{theorem}
For the remainder of the section, we assume $I, g,$ and $v$ are
as in the statement of \Cref{thm:exp_int}. By rescaling the interval $I$
by $t \mapsto \sqrt{\om}t,$
we may assume $g''(0) = -2$ 
so that $g(t) = e^{-t^2} + O(t^3).$
This makes the analysis easier and it only multiplies the integral by a
constant factor of $\sqrt{\om}.$
Therefore, unless otherwise noted, assume $g''(0) = -2.$

\begin{lemma} \label{lemma:subinterval}
    Suppose $J$ is a subinterval of $I,$ also centered at zero.
    Then for any $k \in \N$,
    \[
        \int_I [g(t)]^n v(t) \, \diff t
        = 
        \int_J [g(t)]^n v(t) \, \diff t
        + O(n^{-k}).
    \] \end{lemma}
\begin{proof}
    By assumption, $\max \{ |g(t)| : t \in I \sans J \} < 1$, and so the
    difference between the two integrals decays rapidly in $n.$
\end{proof}
Fix a constant $\tfrac{1}{3} < \al < \tfrac{1}{2}$ and for $n \in \bbN$
define the interval $I_n = [-n^{-\al}, n^{-\al}].$

\begin{lemma} \label{lemma:intin}
    For any $k \in \N$,
    \[
        \int_I [g(t)]^n v(t) \, \diff t
        = 
        \int_{I_n} [g(t)]^n v(t) \, \diff t
        + O(n^{-k}).
    \] \end{lemma}
\begin{proof}
    By \Cref{lemma:subinterval}, we are free to restrict $I$ to a subinterval.
    Therefore we may assume that $|g(t)| < \exp(-\tfrac{1}{2} t^2)$
    holds for all $t \in I.$
    Then
    \[
        |g(t)|^n \le \exp(-\tfrac{1}{2} n t^{2})
        \le \exp(-\tfrac{1}{2} n^{1-2\al})
    \]
    for any $t \in I$ with $|t| > n^{-\al}.$
    Since $1 - 2 \al > 0,$ the result follows.
\end{proof}
\begin{lemma} \label{lemma:gamma}
    For any integer $k \ge 0$ and even integer $j \ge 0,$
    \[
        \int_{I_n} e^{-n t^2} t^j \, \diff t
        \ = \ 
        n^{-(j+1)/2} \ \Gam(\tfrac{j+1}{2}) \ + \ O(n^{-k})
    \]
    where $\Gam$ is the standard gamma function.
    If $j$ is odd, the integral is zero.
\end{lemma}
\begin{proof}
    Write $I_n = [-n^{-\al}, 0] \cup [0, n^{-\al}].$
    We only consider the interval [0, $n^{-\al}$] as the other interval
    is handled similarly.
    Using the substitution $x = n t^2,$
    \[
        \int_{0}^{n^{- \al}} e^{-n t^2} t^j \, \diff t
        \ = \
        \frac{1}{2}
        \int_{0}^{n^{1 - 2 \al}}
        e^{-x} n^{-(j+1)/2} x^{(j-1)/2} \, \diff x.
    \]
    For a fixed $j$ and using that $1 - 2 \al > 0,$
    if $n$ is sufficiently large, then
    $x^{(j-1)/2} < e^{x/2}$ holds for all $x > n^{1 - 2a}.$
    Therefore,
    \begin{align*}    
        \int_{n^{1 - 2 \al}}^{\infty}
        e^{-x} n^{-(j+1)/2} x^{(j-1)/2} \, \diff x \ 
            &< \ \int_{n^{1 - 2 \al}}^{\infty}
                e^{-x/2} n^{-(j+1)/2} \, \diff x \ 
        \\
            &= \ 2 n^{-(j+1)/2}  \exp(-\tfrac{1}{2} n^{1-2 \al})
    \end{align*}
    which is $O(n^{-k})$ for any $k.$
    This shows that
    \[
        \int_{0}^{n^{- \al}} e^{-n t^2} t^j \, \diff t
        \ = \
        \frac{1}{2}
        \int_{0}^{\infty}
        e^{-x} n^{-(j+1)/2} x^{(j-1)/2} \, \diff x
        \ + \ O(n^{-k}).
    \]
    If $j$ is odd, the integral is zero by symmetry.
\end{proof}
We now consider the integrand $[g(t)]^n v(t),$
but only for $n$ and $t$ where $|t| < n^{-\al}.$
Therefore, for the rest of this section,
big-O notation will be defined as follows:
$\Phi(n,t)$ is $O(\Psi(n,t))$
if there is $C > 0$ such that
\[
    |\Phi(n,t)| \le C \Psi(n,t)
\]
for all $n \in \bbN$ and $t \in I$ with $|t| < n^{-\al}.$
Define a smooth function $h : I \to \bbC$ by $g(t) = e^{-t^2} e^{h(t)}.$
This implies that
$h(0) = h'(0) = h''(0) = 0$ and that
\[
    [g(t)]^n v(t) = e^{-n t^2} e^{n h(t)} v(t).
\]
Fix an integer $k.$ 
For reasons explained later, we assume $1-k\alpha < k(1-3\alpha) < 0$.
By Taylor's theorem, we can write $h(t)$ as
$h(t) = p_h(t) + r_h(t)$
where $p_h(t)$ is a polynomial with degree less than $k$
and
$r_h(t)$ is $O(t^k) = O(n^{-k \al}).$
Similarly, write $v(t) = p_v(t) + r_v(t).$
For the exponential function, write
\[
    e^z = \exp(z) = p_e(z) + r_e(z)
    \quad
    \text{where}
    \quad
    p_e(z) = 1 + z + \frac{1}{2!} z^2 + \cdots + \frac{1}{(k-1)!}z^{k-1}.
\]
The remainder term satisfies $|r_e(z)| \le e |z|^k$
for all $z \in \bbC$ with $|z| < 1.$
Now
\[
    \exp(n h(t)) = \exp(n p_h(t)) \exp(n r_h(t))
\]
where $n r_h(t)$ is $O(n t^k) = O(n^{1 - k \al}).$
Since $1 - k \al$ is negative,
$n r_h(t)$ tends uniformly towards zero as $n$ increases
and so
\[
    \exp(n r_h(t)) = 1 + O(n^{1 - k \al}).
\]
Consider
\[
    \exp(n p_h(t)) = p_e(n p_h(t)) + r_e(n p_h(t)).
\]
Since 
$h(0) = h'(0) = h''(0) = 0,$
it follows that $p_h(t)$ is $O(t^3)$ and so
$n p_h(t)$ is $O(n t^3) = O(n^{1 - 3 \al}).$
By assumption, $1 - 3 \al < 0.$ Hence, $n p_h(t)$ is small when $n$ is large and
\[
    r_e(n p_h(t)) = O([n^{1-3\al}]^k) = O(n^{k(1-3\al)}).
\]
Using the assumption $1-k\alpha<k(1-3\alpha)<0$, 
the above estimates combine to show
\begin{align*}
    e^{n h(t)} v(t)
        &= \big [p_e(n p_h(t)) + r_e(n p_h(t)) \big ] \
            e^{n r_h(t)} \
            \big [p_v(t) + r_v(t) \big ]
    \\
        &= \big [p_e(n p_h(t)) + O(n^{k(1-3\al)}) \big ] \
            \big [1 + O(n^{1 - k \al}) \big ] \
            \big [p_v(t) + O(n^{-k\al}) \big ]
    \\
        &= p_e(n p_h(t)) p_v(t)
            \ + \ O(n^{k(1-3\al)}).
\end{align*}

The expression $p_e(n p_h(t)) p_v(t)$ is a polynomial in both
$n$ and $t,$ which we write as a finite sum $\sum_{i,j} a_{ij} n^i t^j$,
where $1\leq i \leq k-1$ and $0 \leq j \leq k(k-1)$.

\begin{lemma} \label{lemma:threeij}
    If the coefficient $a_{ij}$ is non-zero, then $3i \le j.$
\end{lemma}
\begin{proof}
    Recall that $p_h(t)$ only has terms of degree 3 and higher.
    This means for any $m \ge 0$ that $[n p_h(t)]^m$ is a polynomial in $n$ and $t$
    whose coefficient for $n^i t^j$ is zero if $3i > j.$
    Summing these, we see that
    \[
        p_e(n p_h(t)) = \sum_{m=0}^{k-1} \frac{1}{m!} [n p_h(t)]^m,
    \]    
    also satisfies the same property.
    Multiplying by $p_v(t)$ can only increase the powers of $t,$
    not those of $n.$
\end{proof}

By \Cref{lemma:gamma},
\begin{align*}
    \int_{I_n} e^{-n t^2} e^{n h(t)} v(t) \, \diff t\
        &= \ \int_{I_n} e^{-n t^2} \sum_{i,j} a_{ij} n^i t^j \, \diff t 
            \ + \ O(n^{k(1-3\al)}).
    \\
        &= \sum_{ \substack{i,j \\ j\ \text{even}} }
            a_{ij} n^{i-(j+1)/2} \Gamma(\tfrac{j+1}{2})
            \ + \ O(n^{k(1-3\al)}).
\end{align*}
The conditions that $j$ is even and $3i \le j$ ensure that 
\[
    i - \tfrac{j+1}{2}
    \ \in \
    \left \{ \ 
    -\tfrac{1}{2},\ -\tfrac{3}{2},\ -\tfrac{5}{2},\ -\tfrac{7}{2},\ \ldots \
    \right \}.
\]
Therefore, there are constants $c_1, c_3, c_5, \ldots, c_m$
such that
\[
    \int_I [g(t)]^n v(t) \, \diff t
    \ = \
    c_1 n^{-1/2} + c_3 n^{-3/2} + c_5 n^{-5/2} + \cdots + c_m n^{-m/2}
    \ + \ O(n^{k(1-3\al)}).
\]
We can freely assume here that $m/2 < k(3\al-1)$,
since any terms beyond this can be absorbed into the $O(n^{k(1-3\al)})$
contribution.
\begin{lemma} \label{lemma:hdepend}
    The constant $c_j$ only depends on the derivatives
    \[
        h(0), h'(0), h''(0), \ldots, h^{(j+1)}(0)
        \qandq
        v(0), v'(0), v''(0), \ldots, v^{(j)}(0).
    \] \end{lemma}
\begin{proof}
    For an integer $0 \le j < k,$
    consider the term $\frac{1}{j!} h^{(j)}(0) t^j$
    which occurs in the Taylor polynomial $p_h(t).$
    This yields a term $\frac{1}{j!} h^{(j)}(0) n t^j$
    in $n p_h(t)$ which contributes to the term for
    $n t^j$ in $p_e(n p_h(t)).$
    (In fact, it also contibutes to the terms
    $n^2 t^{2j}$, $n^3 t^{3j}$, et cetera,
    but these have a smaller impact on the final estimate.)

    If we multiply this term for $n t^j$ by the constant term of $v(0)$ in $p_v,$
    we see that $h^{(j)}(0)$ affects the value
    of the $a_{1j}$ coefficient.
    This means its impact on the final estimate for the integral is
    of the order
    $n^{1 - (j+1)/2}$
    = $n^{-(j-1)/2}.$
    Hence, $h^{(j)}(0)$ only affects those constants $c_{\ell}$
    with $\ell \ge j-1.$
    A similar analysis shows that a change to $v^{(j)}(0)$ 
    will result in a change to the integral on the order of
    $n^{-(j-1)/2}.$
\end{proof}
\begin{corollary} \label{cor:gdepend}
    The constant $c_j$ only depends on the derivatives
    \[
        g(0), g'(0), g''(0), \ldots, g^{(j+1)}(0)
        \qandq
        v(0), v'(0), v''(0), \ldots, v^{(j)}(0).
    \] \end{corollary}
\begin{proof}
    Since $g(t) = e^{-t^2} e^{h(t)},$ one can show that
    $h^{(j)}(0)$ only depends on
    \[
        g(0), g'(0), g''(0), \ldots, g^{(j)}(0).
        \qedhere
    \] \end{proof}

At the start of this section, we had a function $g$ with $g''(0) = -2 \om < 0$
and we rescaled $t \mapsto \sqrt{\om} t$ in order to assume $g''(0) = -2$
so that $g(t) = e^{-t^2} + O(t^3).$
This has the effect of scaling the integral by a constant factor of $\sqrt{\om}.$
Thus in general, the formula for $c_1$ is
\[
    c_1 = \frac{v(0)}{\sqrt{\om}} \Gamma(\tfrac{1}{2})
    = \sqrt{\frac{\pi}{\om}} v(0)
\]
where again $g''(0) = -2 \om.$

To estimate the integral to a given order, consider an integer $L>3$ and set 
$k=2L$ and $\alpha =\frac{L-1}{k}$.
Then $1-k\alpha = 2-L <3-L = k(1-3\alpha)$ and the above results give an estimate
\[
    \left|
    c_1 n^{-1/2} + c_3 n^{-3/2} + c_5 n^{-5/2} + \cdots + c_m n^{-m/2}
    \ - \ 
    \int_I [g(t)]^n v(t) \, \diff t
    \right|
    \ < \ R n^{3-L}
\]
where $R > 0$ depends only on the derivatives of $g$ and $v$ up to order $k=2L$.
This completes the proof of \Cref{thm:exp_int}, where the expression $2k+6$ 
in the statement corresponds to $2L=2(L-3)+6$ here.

\section{Transfer operators and smooth curves of Lipschitz functions}\label{Sec:4}

This section focuses on the non-invertible case; however, for the sake of notation, we suppress the $+$ sign on both $\sigma$ and $F$. 
We first define the twisted transfer operators associated to the one-sided shift and recall some known results. In \S\ref{sec:4.2}, we study the action of the twisted transfer operators to families of elements in $\mathscr{F}_\theta$.

\subsection{Twisted transfer operators}\label{sec:tto}

The \emph{transfer operator} $L_{\sigma} \colon L^1(X, \mu) \to L^1(X, \mu)$ for $\sigma$ is the bounded operator defined by 
\[
(L_\sigma w)(x) = \sum_{\sigma y=x} e^{u(y)} w(y),
\]
where, we recall, $u \colon X \to \R$ is the potential associated to the Gibbs measure $\mu$.
Note that $\overline{L_\sigma w} = L_\sigma \overline{w}$ and $L_\sigma 1 = 1$. 
For all $v \in L^\infty(X)$ and $w \in L^1(X)$, the operator $L_\sigma$ satisfies 
\[
\int_X (v \circ \sigma) \, \overline{w} \diff \mu = \int_X v \cdot (L_\sigma \overline{w}) \diff \mu.
\]
Similarly, we can define the transfer operator $L_F \colon L^1(X \times \R, \nu) \to  L^1(X \times \R, \nu)$ for the skew product $F$ by 
\[
\int_{X\times \R} (r \circ F) \, \overline{s} \diff \nu = \int_{X\times \R} r \cdot (L_F \overline{s}) \diff \nu,
\]
for every $s \in L^1(X \times \R, \nu)$ and $r \in  L^\infty(X \times \R, \nu)$;
explicilty, 
\[
(L_F s)(x,t) = \sum_{\sigma y=x} e^{u(y)} s(y, t - f(y)).
\]

For any $z \in \C$, let us further introduce the \emph{twisted transfer operator} $\mathcal{L}_z \colon L^1(X,\mu) \to L^1(X,\mu)$, which we define by
\[
(\mathcal{L}_zw)(x) = \sum_{\sigma y=x} e^{u(y)-izf(y)} w(y).
\]
By induction, it is easy to see that the $n$-th iterate of $\mathcal{L}_z$ is given by 
\[
(\mathcal{L}^n_z w)(x)=\sum_{\sigma^n y = x} e^{S_nu(y)-iz S_nf(y)} w(y),
\]
where, as before, $S_nu$ and $S_nf$ denote the $n$-th Birkoff sums of $u$ and $f$ respectively.

Recall that both $\mathcal{L}_0$ and its twisted version $\mathcal{L}_z$ restrict to bounded operators on $\mathscr{F}_\theta$, the space of complex valued Lipschitz functions. On this space, the untwisted operator $\mathcal{L}_0$ has 1 as a simple eigenvalue and the rest of its spectrum is contained in a closed disk inside the open ball $\{ z \in \C : |z| <1\}$, see, e.g., \cite[Chapter 2]{PaPo}.
The map $z \mapsto \mathcal{L}_z$ is analytic for $z \in \C$, hence classical results from the theory of analytical perturbations of bounded operators (see, e.g., \cite{Kat}) apply to our setting. In particular, since the operator $\mathcal{L}_0$ has a spectral gap, we can deduce the following result.
\begin{theorem}\label{thm:anal_pert}
There exists $\kappa >0$ such that $\mathcal{L}_z \colon \mathscr{F}_\theta \to \mathscr{F}_\theta$ has a spectral gap for all $|z|\leq \kappa$.
Moreover, there exist $\lambda_z \in \C$ and linear operators $\mathcal{P}_z, \mathcal{N}_z$ on $\mathscr{F}_\theta$ such that $\mathcal{L}_z = \lambda_z \mathcal{P}_z + \mathcal{N}_z$ and which satisfy the following properties:
\begin{itemize}
\item $\lambda_z, \mathcal{P}_z$ and $\mathcal{N}_z$ are analytic on the disk $\{ z \in \C : |z| \leq \kappa \}$,
\item $\mathcal{P}_z$ is a projection and its range is one dimensional,
\item $\mathcal{P}_z \circ \mathcal{N}_z = \mathcal{N}_z\circ \mathcal{P}_z = 0$,
\item the spectral radius $\rho(\mathcal{N}_z)$ of $\mathcal{N}_z$  satisfies $\rho(\mathcal{N}_z) < \lambda_z - \delta$ for some $\delta>0$ independent of $z$.
\end{itemize}
\end{theorem}
From the previous theorem it follows immediately that there exists $0<\rho<1$ and, for every $\ell \geq 0$, there exist constants $K_\ell>0$ such that 
\begin{equation}\label{eq:estnorms}
\max \left\{ \|\mathcal{P}^{(j)}_z \|_\theta, \|\mathcal{N}_z^{(j)}\|_\theta : 0\leq j\leq \ell \right\} \leq K_\ell, \text{\ \ \ and\ \ \ } \|\mathcal{N}_z \|_\theta \leq \rho, 
\end{equation}
where $\mathcal{P}_z^{(j)}$ and $\mathcal{N}_z^{(j)}$ denote the $j$-th derivatives of $\mathcal{P}_z$ and $\mathcal{N}_z$ with respect to $z$, and $\| \cdot\|_{\theta}$ is the operator norm on  $\mathscr{F}_\theta$.
If we restrict to \emph{real} twisting parameters $\xi \in [-\kappa, \kappa]$, we have the following additional result.
\begin{proposition}\label{prop:deriv_eigen}
There exists a positive constant $A_\kappa$ such that for all $\xi \in [-\kappa, \kappa]$ we have
\[
\left\lvert \lambda_\xi - (1-\omega \xi^2)\right\rvert \leq A_\kappa \xi^3,
\]
where
\[
\omega = \lim_{n \to \infty} \frac{1}{2n} \int_X (S_nf)^2 \diff \mu < \infty.
\]
\end{proposition}
The proof of \Cref{prop:deriv_eigen} can be found, e.g., in \cite[Chapter 4]{PaPo} or in \cite[Section 4]{Sar}. 
In particular, up to possibly choosing a smaller $\kappa$, we can assume that $|\lambda_\xi|<1$ for all $|\xi| \in (0,\kappa]$.

\subsection{Smooth curves of Lipschitz functions}\label{sec:4.2}

We apply the transfer operators $\mathcal{L}_\xi$ for real parameters $\xi$ to families of elements of $\mathscr{F}_\theta$ which satisfy the following properties.

\begin{definition}\label{def:good_family}
A \emph{smooth curve} in $\mathscr{F}_\theta$ is a collection $\phi = (\phi_\xi)_{\xi \in \R} \subset \mathscr{F}_\theta$ such that, for every $k \geq 0$,
\begin{enumerate}
\item the $k$-th derivative $\phi^{(k)}_\xi (x)= \partial^{k}_\xi \phi_\xi(x)$ exists for every $x \in X$,
\item $\phi^{(k)}_\xi \in \mathscr{F}_\theta$ and 
\[
\infnorm(\phi,k) := \sup_{\xi \in \R} \, \max_{0 \leq j \leq k} \, \| \phi^{(j)}_\xi \|_{\infty}< \infty,  \qquad \text{and} \qquad M(\phi,k) := \sup_{\xi \in \R} \, \max_{0 \leq j \leq k} \, \| \phi^{(j)}_\xi \|_{\theta} <\infty.
\]
\end{enumerate}
\end{definition}

We use the following notation: when writing $a \ll b$ we mean that there exists a constant $C$ independent of $a$ and $b$ such that $a \leq C b$. In our case, the implicit constant $C$ depends only on the skew product $F$. If we write $a \ll_\ell b$, we mean that the constant $C = C_\ell$ might depend on $\ell$, but not on $a$ or $b$.

In the remainder of this section, we are going to prove the following result.
\begin{proposition}\label{prop:Krick_step2}
Let $\phi, \psi$ be smooth curves in $\mathscr{F}_\theta$. 
For all $x \in X$ and for all $k \in \N$, there exist $c_1(x), c_3(x), \ldots, c_{2k-1}(x) \in \bbC$, with 
$c_1(x) = \sqrt{\frac{\pi}{\om}}  \, \mu(\phi_0) \, \overline{\psi_0(x)} $ and 
\[
|c_j| \ll_j \infnorm(\psi, j) \cdot M(\phi, j),
\] 
such that for every $n \in \N$ we have
\[
\begin{split}
&\Bigg\lvert c_1(x) n^{-1/2} + c_3(x) n^{-3/2} + \cdots + c_{2k-1}(x) n^{-k+ 1/2} \ - \ \int_\R \overline{\psi_\xi(x)}  \cdot (\mathcal{L}_\xi^n \phi_\xi)(x)  \, \diff \xi \Bigg\rvert \\
&\qquad \qquad \ll_k \left( \sup_{x \in X} \int_\R |\phi_\xi(x)| \, \diff \xi \right) \cdot  \sup_{|\xi| \geq \kappa} \|\mathcal{L}_\xi^n \phi_\xi \|_{\infty} + \infnorm(\psi, 2k+6) \cdot M(\phi, 2k+6) \cdot n^{-k}.
\end{split}
\]
\end{proposition}

Fix two smooth curves $\phi, \psi$ as above. For every $x \in X$, we define
\[
b_x(\xi) = \overline{\psi_\xi(x)} \cdot (\mathcal{P}_\xi \phi_\xi)(x), \qquad \text{ for $\xi \in [-\kappa, \kappa]$.}
\]

\begin{lemma}\label{lemma:estimate_bx}
For any $x \in X$, the function $b_x$ is smooth on $[-\kappa, \kappa]$ and for every $k \geq 0$ we have
\[
\max_{x \in X}\|b_x\|_{\mathscr{C}^k} \ll_k \infnorm(\psi, k) \cdot M(\phi, k).
\]
Moreover, $b_x(0) = \mu(\phi_0) \,  \overline{\psi_0(x)}$.
\end{lemma}
\begin{proof}
Each term in the definition of $b_x$ is smooth by definition and by the analytic perturbation result, Theorem \ref{thm:anal_pert}, applied to $\mathcal{L}_\xi$. For any $j \geq 0$ and for every $\xi \in [-\kappa, \kappa]$, we have
\[
b_x^{(j)} (\xi) = \sum_{j_1+j_2+j_3 = j} \, \frac{j!}{j_1! \, j_2! \, j_3!} \,  \overline{\psi^{(j_1)}_\xi(x)} \cdot  \mathcal{P}_\xi^{(j_2)} (\phi^{(j_3)}_\xi)(x),
\]
thus, we can estimate
\begin{equation*}
\begin{split}
| b_x^{(j)} (\xi) | &\ll_j \sum_{j_1+j_2+j_3 = j} |\psi^{(j_1)}_\xi(x)| \cdot | \mathcal{P}_\xi^{(j_2)} (\phi^{(j_3)}_\xi)(x)| \\
& \leq \sum_{j_1+j_2+j_3 = j} \| \psi^{(j_1)}_\xi \|_{\infty} \cdot  \|\mathcal{P}_\xi^{(j_2)}\|_{\theta} \cdot  \| \phi^{(j_3)}_\xi \|_{\theta} \\
& \ll_j \left(\max_{0\leq i \leq j}\|\mathcal{P}_\xi^{(i)}\|_{\theta} \right) \, \infnorm(\psi, k) \cdot M(\phi, k).
\end{split}
\end{equation*}
Summing over all $0\leq j \leq k$ proves the first part. The final claim follows from the fact that $\mathcal{P}_0\phi_0(x) = \int_X \phi_0(x) \diff \mu(x)$.
\end{proof}

From Lemma \ref{lemma:estimate_bx} and Theorem \ref{thm:exp_int}, we deduce the following result.

\begin{lemma}\label{thm:stationary_phase_lambda}
For any $x \in X$ and any $k \in \N$, there are constants
    $c_1(x), c_3(x), \ldots, c_{2k-1}(x) \in \bbC$ and $R > 0$
    such that
    \[
        \left|
        c_1(x) n^{-1/2} + c_3(x) n^{-3/2} + \cdots + c_{2k-1}(x) n^{-k+ 1/2}
        \ - \ 
        \int_{-\kappa}^\kappa \lambda_\xi^n \, b_x(\xi) \, \diff \xi
        \right|
        \ < \ R n^{-k}
    \]
    for all $n \in \N.$
The constant $c_1(x)$ is given by       \begin{math}
            \begin{displaystyle}
                c_1(x) = \sqrt{\frac{\pi}{\om}}  \, \mu(\phi_0)  \, \overline{\psi_0(x)} \end{displaystyle} \end{math},
 and the other constants $c_j(x)$ and $R$ satisfy 
\[
|c_j| \ll_j \infnorm(\psi, j) \cdot M(\phi, j),
\qquad
\text{and}
\qquad 
|R| \ll_k  \infnorm(\psi, 2k+6) \cdot M(\phi, 2k+6).
\]
\end{lemma}
\begin{proof}
We apply Theorem \ref{thm:exp_int} with $g(\xi) = \lambda_\xi$, which is analytic for $|\xi| < \kappa$. The assumptions are satisfied since $0$ is a critical point, $\lambda_0=1$, and 
\[
\frac{\diff^2}{(\diff \xi)^2} \Big\vert_{\xi = 0} \lambda_\xi = - 2 \omega<0.
\]
The estimates on the terms $c_j(x)$ and $R$ follow from Lemma \ref{lemma:estimate_bx}. The implicit constants depend on the derivatives on $\lambda$, and hence on the skew product only.
\end{proof}

We are now ready to prove Proposition \ref{prop:Krick_step2}.
\begin{proof}[{Proof of Proposition \ref{prop:Krick_step2}}]
Splitting the domain of integration into the interval $[-\kappa, \kappa]$ and its complement, we get
\[
\begin{split}
&\left\lvert \int_\R \overline{\psi_\xi(x)}  \cdot (\mathcal{L}_\xi^n \phi_\xi)(x)  \, \diff \xi - \int_{-\kappa}^\kappa \overline{\psi_\xi(x)} \cdot (\mathcal{L}_\xi^n \phi_\xi)(x) \diff \xi \right\rvert \\
&\qquad \qquad \leq \int_{|\xi| \geq \kappa} |\psi_\xi(x) | \cdot |(\mathcal{L}_\xi^n \phi_\xi)(x) |\diff \xi \leq  \left( \sup_{x \in X} \int_\R |\phi_\xi(x)| \, \diff \xi \right) \cdot  \sup_{|\xi| \geq \kappa} \|\mathcal{L}_\xi^n \phi_\xi \|_{\infty}.
\end{split}
\]
We now consider the integral in the left hand-side above: for $|\xi|\leq \kappa$, we can decompose the twisted transfer operator into $\mathcal{L}_\xi^n = \lambda^n_\xi \mathcal{P}_\xi + \mathcal{N}_\xi^n$. By Theorem \ref{thm:anal_pert}, we obtain
\[
\left\lvert (\mathcal{L}_\xi^n \phi_\xi)(x) - \lambda_\xi^n (\mathcal{P}_\xi \phi_\xi)(x) \right\rvert \leq \rho^n \|\phi_\xi\|_\theta \leq  M(\phi,0) \, \rho^n\ll_k M(\phi,0)\, n^{-k},
\]
therefore, 
\[
\begin{split}
\left\lvert \int_\R \overline{\psi_\xi(x)}  \cdot (\mathcal{L}_\xi^n \phi_\xi)(x)  \, \diff \xi - \int_{-\kappa}^\kappa \lambda_\xi^n b_x(\xi)\right\rvert \ll_k & \infnorm(\psi,0) \, M(\phi,0) \, n^{-k} \\
&+ \left( \sup_{x \in X} \int_\R |\phi_\xi(x)| \, \diff \xi \right) \cdot  \sup_{|\xi| \geq \kappa} \|\mathcal{L}_\xi^n \phi_\xi \|_{\infty}.
\end{split}
\]
The conclusion follows now from Lemma \ref{thm:stationary_phase_lambda}.
\end{proof}

\section{Fourier transforms of the observables}\label{Sec:z}

In this section, we show that the Fourier transform of elements of $\mathscr{L} =\mathscr{L}^+$ can be seen as smooth curves in $\mathscr{F}_\theta$. We remark that the results we prove in this section still hold if we replace $\theta$ with any fixed $\beta \in [\theta, 1)$.

\subsection{Fourier transforms}\label{sec:observ}

Recall that the Fourier transform $\widehat{v}$ of a function $v \in \mathscr{S}$ is defined by 
\[
\widehat{v}(\xi) = \int_{-\infty}^{\infty} v(t) e^{- i t\xi} \diff t.
\] 
\begin{lemma}\label{lem:L1_norm_FT}
The Fourier transform of $v \in \mathscr{S}$ satisfies
\[
\| \widehat{v}\|_\infty \leq 4 \, \|v\|_{2,0}, \quad \text{ and } \quad \| \widehat{v}\|_{L^1} \leq 16 \, \|v\|_{2,2}.
\]
\end{lemma}
\begin{proof}
We have
\begin{equation*}
\| \widehat{v}\|_\infty \leq \int_{-\infty}^\infty |v(t)| \diff t = \int_{-1}^1 |v(t)| \diff t + \int_{\R \setminus [-1,1]} \frac{1}{t^2}|t^{2} \cdot v(t)| \diff t \leq 2(\|v\|_{\infty} + \|t^{2} \cdot v\|_{\infty}),
\end{equation*}
which proves the first claim.
By the properties of the Fourier transform, we have 
\begin{equation*}
\begin{split}
\| \widehat{v}\|_{L^1} & = \int_{-1}^1 |\widehat{v}(\xi)| \diff \xi + \int_{\R \setminus [-1,1]} \frac{1}{\xi^2}|\xi^{2} \cdot \widehat{v}(\xi)| \diff \xi \leq 2( \|\widehat{v}\|_{\infty} + \|\xi^{2} \cdot \widehat{v}\|_{\infty}) = 2(\|\widehat v\|_{\infty} + \| \widehat{v''}\|_{\infty}),
\end{split}
\end{equation*}
and now we apply the estimate obtained before on the two summands.
\end{proof}

Let $s \in \mathscr{L}^{+}$. From now on, for each point $x \in X$, let $s_x$ denote the Schwartz function $s(x) \in \mathscr{S}$. Note that its Fourier transform $\widehat{s_x}$ is again a Schwartz function.
For any fixed $\xi \in \R$, we define the function $\phi_\xi \colon X \to \C$ as
\[
\phi_\xi(x) = \widehat{s_x}(\xi).
\]
We now show that $\phi \colon \xi \mapsto \phi_\xi$ is a smooth curve in $\mathscr{F}_\theta$ according to Definition \ref{def:good_family}. More in general, we prove the following result.

\begin{proposition}\label{prop:FT_smooth_curves}
For any $m \geq 0$, the family $\Phi = (\Phi_\xi)_{\xi \in \R}$ given by $\Phi_\xi = \mathcal{L}^m_\xi \phi_\xi$ is a smooth curve in $\mathscr{F}_\theta$, where 
\[
\infnorm(\Phi,k) \ll_k (1+m)^{k+2} \|s\|_{\infty, k+2,0}, \qquad \text{and} \qquad  M(\Phi,k) \ll_k (1+m)^{k+2} ( \|s\|_{\infty, k+2,1} + \theta^m  |s|_{\infty, k+2,0}).
\]
\end{proposition}

Let us prove first some preliminary estimates.

\begin{lemma}\label{lem:philip}
For all fixed $\xi \in \R$, we have $\phi_\xi \in \mathscr{F}_\theta$. Moreover, for all $\ell \geq 0$, 
\[
\|\partial^\ell_\xi \phi_\xi\|_\infty \leq 4 \, \infldz{s} \text{\ \ \ and\ \ \ } | \partial^\ell_\xi \phi_\xi|_\theta \leq  4\, \ltldz{s}.
\]
\end{lemma}
\begin{proof}
For all $x \in X$, by Lemma \ref{lem:L1_norm_FT},
\begin{equation*}
\begin{split}
|\partial^\ell_\xi \phi_\xi(x)| &= |\partial^\ell_\xi \widehat{s_x}(\xi)| \leq \|\widehat{t^\ell s_x(t)}\|_{\infty} \leq 4 \, \|t^\ell s_x(t)\|_{2,0} \leq 4  \, \|s_x\|_{\ell+2,0} \leq 4 \, \infldz{s}.
\end{split}
\end{equation*}
Similarly, for all $x,y \in X$, 
\begin{equation*}
\begin{split}
|\partial^\ell_\xi\phi_\xi(x) - \partial^\ell_\xi \phi_\xi(y)| &= |\partial^\ell_\xi(\widehat{s_x}-\widehat{s_y})(\xi)| \leq \|\widehat{ [t^\ell (s_x-s_y)]} \|_\infty \leq  4 \, \|s_x - s_y\|_{\ell+2,0} \\
& \leq 4\, \ltldz{s} \, d_\theta(x,y),
\end{split}
\end{equation*}
which proves that $| \partial^\ell_\xi \phi_\xi|_\theta \leq  4\, \ltldz{s}$.
\end{proof}

Notice that the previous lemma proves \Cref{prop:FT_smooth_curves} in the case $m=0$.

\begin{lemma}\label{lem:good_fam_est}
For any $\xi \in \R$, we have
\[
|\xi| \cdot \|\partial^\ell_\xi \phi_\xi\|_\infty \ll_\ell \|s\|_{\infty, \ell+2,1}.
\]
\end{lemma}
\begin{proof}
We notice that 
\[
|\xi|\cdot |\partial^\ell_\xi \phi_\xi(x)| = | \widehat{[\partial_t (t^\ell s_x(t))]}| \leq \ell | \widehat{t^{\ell-1} s_x(t)}| + | \widehat{t^\ell \partial_t s_x(t)}|.
\]
From this point, the proof continues exactly as in Lemma \ref{lem:philip}.
\end{proof}

With the same argument, we can also prove the following fact.
\begin{lemma}\label{lem:philip_2}
Let $\xi >0$. For any $k \in \Z_{\geq 0}$, we have  
\[
\|\phi_\xi\|_\infty \leq 4 \, |\xi|^{-k} \, \|s\|_{\infty, 2, k} \text{\ \ \ and\ \ \ } |\phi_\xi|_\theta \leq  4 \, |\xi|^{-k} \, |s|_{\theta, 2, k}.
\]
\end{lemma}
\begin{proof}
For all $x \in X$, using again Lemma \ref{lem:L1_norm_FT},
\begin{equation*}
\begin{split}
|\phi_\xi(x)| &= |\xi|^{-k} |\widehat{\partial^k s_x}(\xi)| \leq 4 \, |\xi|^{-k} \, \|\partial^k  s_x\|_{2,0} \leq \, |\xi|^{-k} \, \|\partial^k  s\|_{\infty, 2,0}.
\end{split}
\end{equation*}
Similarly, for all $x,y \in X$, 
\begin{equation*}
\begin{split}
|\phi_\xi(x) - \phi_\xi(y)| &=  |\xi|^{-k}  |(\widehat{\partial^k s_x}-\widehat{\partial^k s_y})(\xi)| \leq  4 \,  |\xi|^{-k} \, \|\partial^k (s_x - s_y)\|_{2,0} 
\leq 4 \,  |\xi|^{-k} \,  |s|_{\theta, 2, k} \, d_\theta(x,y),
\end{split}
\end{equation*}
which proves that $|\phi_\xi|_\theta \leq 4 \,  |\xi|^{-k} \, |s|_{\theta, 2, k}$.
\end{proof}
Recall that  $S_m f \colon X \to \R$ denotes the Birkhoff sum
\[
    S_m f(x) = f(x) + f(\sigma(x)) + \cdots + f(\sigma^{m-1}(x)).
\]
\begin{lemma} \label{lemma:birknorm}
    For $m \ge 1$, the Birkhoff sum $S_m f \colon X \to \R$ satisfies
    \[
        \| S_m f \| _\infty \ll m
        \qandq
        |S_m f|_\theta \ll \theta^{-m}.
    \] 
\end{lemma}
\begin{proof}
    From the definition,
    it follows that $ \| S_m f \| _\infty \le m \| f \| _\infty$ 
    and the $ \| f \| _\infty$ can be absorbed by the $\ll$ notation.
    The shift map $\sigma \colon X \to X$ is Lipschitz with respect to $d_\theta$ with constant
    $\theta^{-1}.$
    Therefore,
    $|f \circ \sigma^n|_\theta \le \theta^{-n} |f|_\theta$ for $0 \le n < m,$ and
    \[
        |S_m f|_\theta \ \le 
        \ \left( 1 + \theta^{-1} + \cdots + \theta^{-(m-1)} \right) \ |f|_\theta
        \ \le
        \ \frac{|f|_\theta}{\theta^{-1} - 1} \ \theta^{-m}.
        \qedhere
    \] \end{proof}

We also note the following simple fact, which can be obtained by spelling out the necessary definitions.

\begin{lemma}\label{lemma:easyfact}
For any functions $a,b \in \mathscr{F}_{\theta}$, we have
\[
|a \cdot b|_{\theta} \leq |a|_{\theta}  \, \| b\|_{\infty}  + |b|_{\theta} \, \|a \|_{\infty}. 
\]
In particular, for every $n \in \N$, 
\[
|a^n|_{\theta} \leq n |a|_{\theta} \, \|a\|^{n-1}_{\infty}. 
\]
\end{lemma}

Let us now turn to the proof of \Cref{prop:FT_smooth_curves}.

\subsection{Proof of \Cref{prop:FT_smooth_curves}}

Since we already proved the case $m=0$  in \Cref{lem:philip}, let us assume that $m \geq 1$. By the definition of the twisted transfer operator, we can write
\begin{equation}\label{eq:psiellxi}
\begin{split}
\partial^{k}_\xi \Phi_\xi (x) &= \partial_\xi^k \left[ \sum_{\sigma^my = x}  e^{S_mu(y) - i \xi S_mf(y)} \widehat{s_y}(\xi) \right] \\
&= \sum_{j+\ell=k} \frac{k!}{j! \, \ell!}  \sum_{\sigma^my = x}  e^{S_mu(y)} [\partial_\xi^j e^{ - i \xi S_mf(y)}] \cdot [\partial_\xi^\ell \widehat{s_y}(\xi)] \\
&= \sum_{j+\ell=k} \frac{k!}{j! \, \ell!}  \sum_{\sigma^my = x}  e^{S_mu(y)} (- i S_mf(y))^j e^{ - i \xi S_mf(y)} \, \partial^\ell_\xi \phi_\xi(y) \\
&= \sum_{j+\ell=k} \frac{k!}{j! \, \ell!} \mathcal{L}^m_\xi [(- i S_mf)^j \, \partial^\ell_\xi \phi_\xi](x).
\end{split}
\end{equation}

\begin{lemma}\label{lemma:thetanormpsi}
For all $\xi \in \R$, we have
\[
\begin{split}
&\|\partial_\xi^k\Phi_\xi\|_\infty  \ll_k \|s\|_{\theta, k+2,0} \, m^{k+2}, \qandq |\partial_\xi^k\Phi_\xi|_{\theta} \ll_k m^{k+2} ( \|s\|_{\infty, k+2,1} + \theta^m |s|_{\theta,k +2,0}).
\end{split}
\]
\end{lemma}
\begin{proof}
Let us fix $\xi \in \R$. 
The expression \eqref{eq:psiellxi} yields
\[
\begin{split}
\|\partial_\xi^k\Phi_\xi\|_\infty  &\ll_k \max \Big\{ \| \mathcal{L}^m_\xi [(- i S_mf)^j \, \partial_\xi^\ell \phi_\xi] \|_{\infty} : j+ \ell = k \Big\} \\
&\ll \max \Big\{ \| S_mf\|_\infty^j  \cdot \| \partial_\xi^\ell \phi_\xi \|_{\infty} : j+ \ell = k \Big\} \ll_\ell \|s\|_{\theta, k+2,0} \, m^{k+2},
\end{split}
\]
where we used \Cref{lem:philip} and the estimate $\| S_mf\|_\infty^j \ll_j m^j$ which follows from \Cref{lemma:birknorm}.

For the second estimate, from the Basic Inequality (see \cite[Proposition 2.1]{PaPo}) it follows that
\[
\begin{split}
|\partial_\xi^k\Phi_\xi|_{\theta} &\ll_k \max \Big\{ \Big\vert \mathcal{L}^m_\xi [(- i S_mf)^j \, \partial_\xi^\ell\phi_\xi] \Big\vert_{\theta} : j+\ell = k \Big\} \\
& \ll  \max \Big\{  \theta^m \Big\vert (- i S_mf)^j \, \partial_\xi^\ell \phi_\xi \Big\vert_{\theta} + \xi \| (- i S_mf)^j \, \partial_\xi^\ell\phi_\xi \|_{\infty} : j+ \ell = k \Big\}.
\end{split}
\]
Using \Cref{lemma:easyfact}, we obtain
\[
|\partial_\xi^k\Phi_\xi|_{\theta}\ll_k \max \Big\{  \theta^m \Big[ j |S_mf|_{\theta} m^{j-1} \, \|\partial_\xi^\ell  \phi_\xi\|_\infty +  m^j \, |\partial_\xi^\ell  \phi_\xi|_{\theta} \Big] + \xi m^j \|\partial_\xi^\ell  \phi_\xi \|_{\infty} : j+ \ell = k  \Big\}.
\]
Finally, by \Cref{lemma:birknorm}, \Cref{lem:philip}, and \Cref{lem:good_fam_est}, we conclude
\[
\begin{split}
|\partial_\xi^k\Phi_\xi|_{\theta} &\ll_k \max \Big\{  m^{j-1+\ell +2} \|s\|_{\infty, \ell +2,0} +  \theta^m m^j |s|_{\theta,\ell +2,0}  + m^{j+\ell +2}\|s\|_{\infty, \ell +2,1} : j+\ell =k\Big\} \\
&\ll_k m^{k+2} ( \|s\|_{\infty, k+2,1} + \theta^m |s|_{\theta,k +2,0}),
\end{split}
\]
and the proof is complete.
\end{proof}

\Cref{lemma:thetanormpsi} immediately implies the estimates  
\[
\infnorm(\Phi,k) \ll_k m^{k+2} \|s\|_{\infty, k+2,0}, \qquad \text{and} \qquad  M(\Phi,k) \ll_k m^{k+2} ( \|s\|_{\infty, k+2,1} + \theta^m  |s|_{\infty, k+2,0})
\]
for $m\geq 1$, which therefore completes the proof  of \Cref{prop:FT_smooth_curves}.

\section{Rapid decay}\label{sec:rapid_decay}

In view of Proposition \ref{prop:Krick_step2}, it is crucial to understand the behaviour of the sequence
\[
R_n := \sup_{\xi \geq \kappa} \|\mathcal{L}_\xi^n \phi_\xi \|_{\infty}
\]
where $\phi_\xi$ is the smooth curve in $\mathscr{F}_\theta$ defined by $\phi_\xi(x) = \widehat{s_x}(\xi)$ for a given $s \in \mathscr{L}$.
In our previous work \cite{GHR}, we studied the operator norm of $\mathcal{L}_\xi^n$ as $n\to \infty$; we recall our main result in this direction.
In its simplest form, it says that for every $s \in \mathscr{L}= \mathscr{L}_{\theta}$, the sequence $R_n$ \emph{decays rapidly}, namely for every $\ell \geq 1$ there exists $C_\ell$ such that $R_n \leq C_\ell n^{-\ell}$ for all $n \geq 1$.
For the proof of Theorem \ref{cor:main_results}, we will need a more general and rather technical version, Proposition \ref{prop:rapid_decay} below, where we consider a sequence of functions $(s_m)_{m \geq 0}$ rather than a single $s$. 
Let us first point out that, for any $\beta \in [\theta, 1)$, the space $\mathscr{L}=\mathscr{L}_{\theta}$ associated to the metric $d_{\theta}$ on $X$ is contained in $\mathscr{L}_{\beta}$ (corresponding to $d_{\beta}$) and, for any $p,q \geq 0$ we have
\[
\| s\|_{\beta, p,q} \leq \| s\|_{\theta, p,q}. 
\]

\begin{proposition}\label{prop:rapid_decay}
Let $\beta \in [\theta, 1)$, and let $C>0$. Consider $s \in \mathscr{L}_\theta$, and let $(s_m)_{m \in \N}$ be a sequence in $\mathscr{L}_{\beta}$ such that 
\[
\| s_m\|_{\beta,2,k} \leq C m^2 \beta^{-m} \| s \|_{\theta,2,k+3}, \quad \text{for every $k \in \N$.}
\]
Assume that the function $f$ defining the skew product $F(x,t) = (\sigma x,t+f(x))$ has the collapsed accessibility property. Then, for every $c>0$, the sequence $(R_n)_{n \in \N}$ defined by
\[
R_n := \sup \left\{ \|\mathcal{L}_\xi^n (\phi_m)_\xi \|_{\infty} : |\xi| \geq \kappa,\ m < c \log n\right\}, \quad \text{where} \quad  (\phi_m)_{\xi}(x) := \widehat{(s_m)_x}(\xi),
\]
decays rapidly: for every $\ell \geq 1$ we have
\[
R_n \ll_\ell \, c^2 \, \| s\|_{\theta,2, B(\ell,c)} \, n^{-\ell}, \quad \text{where} \quad B(\ell,c) \ll_\beta \ell + c.
\]
\end{proposition}

Let us explain how Proposition \ref{prop:rapid_decay} can be proved from the results in  \cite{GHR}. 
We are going to use the following abstract result,
which is a slight generalization of \cite[Proposition 7.2]{GHR}.

\begin{proposition}[{\cite[Proposition 7.2]{GHR}}]\label{prop:abstract_rapid_decay}
Suppose $A,B, \alpha$ and $\kappa$ are positive constants and that $(w_n)_{n\geq 0}$ is a decreasing sequence of functions of the form $w_n \colon [\kappa, \infty) \to [0,1]$. Assume that for every $\ell \geq 1$ there exists a constant $C_\ell$ such that $w_0(\xi) \leq C_\ell \xi^{-\ell}$; assume moreover that 
\[
w_{n+N}(\xi) \leq (1-A\xi^{-\alpha}) w_n(\xi) \qquad \text{for all} \qquad N > B \log (\xi).
\]
Then, the sequence $\{W_n\}_{n \geq 0}$ defined by $W_n = \sup_{\xi \geq \kappa} w_n(\xi)$ decays rapidly: for every $\ell \geq 1$ we have
\[
W_n \ll_\ell C_{B(\ell)} \, n^{-\ell}, \qquad \text{where} \qquad B(\ell) \ll \ell.
\]
\end{proposition}
\begin{proof}
By \cite[Lemma 7.3]{GHR}, there exist constants $D,\gamma >0$ such that $w_{n+K}(\xi) < \frac{1}{e} w_n(\xi)$ for all $K > D \xi^{\gamma}$. Without loss of generality, we can assume that $\gamma$ is a positive integer.

Let us fix $\ell \geq 1$. In order to prove the result, we will show that 
\[
W_n \ll_\ell C_{2 \gamma \ell} \, n^{-\ell}\qquad \text{for all} \qquad n\geq 1.
\] 
Fix $n \geq 1$. Let $K, j$ be the unique integers satisfying $D\sqrt{n} < K \leq D\sqrt{n} +1$ and $\ell \log(n) <j \leq \ell \log(n) +1$.

Let us consider $\xi \in [\kappa, \infty)$. If $\xi \geq n^{\frac{1}{2\gamma}}$, then 
\[
w_n(\xi) \leq w_0(\xi) \leq C_{2\ell \, \gamma} \, \xi^{-2\ell \, \gamma} \leq C_{2\ell \, \gamma} \, n^{-\ell}. 
\]
Otherwise, if $\xi <  n^{\frac{1}{2\gamma}}$, we can bound $K > D\sqrt{n}> D \xi^{\gamma}$. Therefore, for all $n\geq jK$, we have 
\[
w_n(\xi) \leq w_{jK}(\xi) < e^{-j} w_0(\xi) \leq e^{-j} \leq n^{-\ell}.
\]
On the other hand, for all $n <jK$,
\[
w_n(\xi) \leq (jK)^{3\ell} \, n^{-3\ell} \leq (D\sqrt{n} +1)^{3\ell}(\ell \, \log(n) +1)^{3\ell} n^{-3\ell} \ll_\ell n^{-\ell}.
\]
Combining all the previous cases gives us the desired inequality.
\end{proof}

In \cite[\S6]{GHR}, for any fixed $|\xi| \neq 0$, we defined a norm $\| \cdot \|_H$ on $\mathscr{F}_{\beta}$, satisfying
\[
\| v\|_H \leq \|v \|_{\beta} \ll |\xi| \cdot \|v\|_H,
\]
for which we proved the following statement. There exist constants $A, \alpha>0$ such that, for all $\xi \geq \kappa$, we have 
\begin{equation}\label{eq:needed_for_rd}
\| \mathcal{L}_\xi^N \|_H \leq 1-A \xi^{-\alpha}, \qquad \text{for all} \qquad N \gg \, |\log \xi|,
\end{equation}
see, in particular, \cite[Proposition 6.1]{GHR}. Let us define
\[
v_{n,m}(\xi) := \| \mathcal{L}_\xi^n (\phi_m)_\xi \|_H, \qquad \text{and} \qquad w_n(\xi) := \sup_{m \in \Z_{\geq 0}} m^{-2} \beta^{m} v_{n,m}(\xi).
\]

\begin{lemma}\label{lem:rapid_decay_final}
The sequence $(W_n)_{n \geq 0}$ defined by $W_n := \sup_{\xi \geq \kappa} w_n(\xi)$ decays rapidly: for any $\ell \geq 1$ we have
\[
W_n \ll_\ell  \| s \|_{\theta,2,B(\ell)} n^{-\ell}, \qquad \text{where} \qquad B(\ell) \ll \ell.
\]
\end{lemma}
\begin{proof}
We verify the assumptions of Proposition \ref{prop:abstract_rapid_decay} applied to the sequence of functions $(w_n)_{n \geq 0}$ defined above. From \eqref{eq:needed_for_rd} it follows that $w_{n+N}(\xi) \leq (1-A\xi^{-\alpha}) w_n(\xi)$ for all $N \gg \log(\xi)$. By \cite[Lemma 6.2]{GHR}, we have that $w_{n+1}(\xi) \leq w_n(\xi)$, so the sequence $(w_n)_{n \geq 0}$ is decreasing.
Finally, from Lemma \ref{lem:philip_2} and the assumptions in Proposition \ref{prop:rapid_decay}, we obtain
\[
w_0(\xi) \leq \sup_m \, m^{-2} \beta^m \| (\phi_m)_\xi \|_{\beta} \leq 4 |\xi|^{-\ell} \sup_m \, m^{-2} \beta^m \| s_m \|_{\beta,2,\ell} \ll \| s \|_{\theta,2,\ell+3} |\xi|^{-\ell}.
\]
Proposition \ref{prop:abstract_rapid_decay} yields the conclusion.
\end{proof}

We can now prove Proposition \ref{prop:rapid_decay}.

\begin{proof}[{Proof of Proposition \ref{prop:rapid_decay}}]
Since $ v_{n,m}(\xi) \leq m^2 \beta^{-m} w_n(\xi)$, by Lemma \ref{lem:rapid_decay_final}, for any fixed $c>0$ and any $m < c \log n$ we have
\[
\begin{split}
\sup_{\xi \geq \kappa} v_{n,m}(\xi) &\leq m^2 \beta^{-m} \, \sup_{\xi \geq \kappa} w_n(\xi) \ll_\ell (c \log n)^2 n^{-c \log \beta} \| s \|_{\theta,2,B(\ell)} n^{-\ell}\\
&\ll_\ell c^2 n^{-(\ell+c \log \beta-1)}.
\end{split}
\]
Taking the supremum over $m$ completes the proof.
\end{proof}

\section{Proof of Theorem \ref{thm:main_1_sd}}\label{Sec:5}

This section is devoted to the proof of Theorem \ref{thm:main_1_sd}. 

Combining Proposition \ref{prop:Krick_step2} with the results in Section \ref{Sec:z} and Section \ref{sec:rapid_decay}, we deduce the following result.
\begin{lemma}\label{lem:conseq}
There exists a constant $B>0$ such that the following holds. Let $r, s \in \mathscr{L}^{+}$. For any $x \in X$ and for any $k \in \N$, there exist $c_1(x), c_3(x), \ldots, c_{2k-1}(x) \in \bbC$, with 
\[
c_1(x) = \sqrt{\frac{\pi}{\om}} \, \nu(s) \cdot \left( \int_\R \overline{r_x(\xi)} \diff \xi \right).
\]
and 
\[
|c_j(x)| \ll_j \| s\|_{\theta, j+2,0} \cdot \| s\|_{\theta, j+2,0},
\] 
such that for every $n \in \N$ we have
\[
\begin{split}
&\Bigg\lvert c_1(x) n^{-1/2} + c_3(x) n^{-3/2} + \cdots + c_{2k-1}(x) n^{-k+ 1/2} \ - \ \int_\R \overline{\psi_\xi(x)}  \cdot (\mathcal{L}_\xi^n \phi_\xi)(x)  \, \diff \xi \Bigg\rvert \\
&\qquad \qquad \ll_k   \| r \|_{\theta, Bk, Bk} \,  \| s \|_{\theta, Bk, Bk} \,  n^{-k}.
\end{split}
\]
\end{lemma}
\begin{proof}
Let $\psi$ and $\phi$ be the smooth curves in $\mathscr{F}_\theta$ associated to $r$ and $s$ respectively. 
By \Cref{lem:philip}, we have
\[
M(\phi, k) \ll \| s\|_{\theta, k+2,0}, \qquad \text{and} \qquad \infnorm(\psi, k) \ll \| r \|_{\theta, k+2,0}.
\]
Furthermore, from Lemma \ref{lem:L1_norm_FT}, we deduce that for any $x \in X$ the function $\xi \mapsto \psi_\xi(x) = \widehat{r_x}(\xi)$ is integrable and 
\[
\sup_{x \in X} \int_\R |\widehat{r_x}(\xi)| \diff \xi \ll \|r\|_{\infty, 2, 2}.
\]
Proposition \ref{prop:Krick_step2} then yields
\begin{equation}\label{eq:from_p_K2}
\begin{split}
&\Bigg\lvert c_1(x) n^{-1/2} + c_3(x) n^{-3/2} + \cdots + c_{2k-1}(x) n^{-k+ 1/2} \ - \ \int_\R \overline{\psi_\xi(x)}  \cdot (\mathcal{L}_\xi^n \phi_\xi)(x)  \, \diff \xi \Bigg\rvert \\
&\qquad \qquad \ll_k \|r\|_{\infty, 2, 2} \cdot  \sup_{|\xi| \geq \kappa} \|\mathcal{L}_\xi^n \phi_\xi \|_{\infty} +  \| s\|_{\theta, 2k+8,0} \cdot \| r \|_{\theta, 2k+8,0} \cdot n^{-k},
\end{split}
\end{equation}
for some constants $c_j(x)$, where 
\[
c_1(x) = \sqrt{\frac{\pi}{\om}} \, \mu(\phi_0)  \, \overline{\psi_0(x)}= \sqrt{\frac{\pi}{\om}} \, \nu(s) \cdot  \left( \int_\R \overline{r_x(\xi)} \diff \xi \right).
\]
and 
\[
|c_j(x)| \ll_j \| r\|_{\theta, j+2,0} \cdot \| s\|_{\theta, j+2,0}.
\]
Applying \Cref{prop:rapid_decay} to the constant sequence $s_m=s$, we deduce that the sequence
\[
\left(\sup_{|\xi| \geq \kappa} \|\mathcal{L}_\xi^n \phi_\xi \|_{\infty}\right)_{n \geq 0}
\]
decays rapidly; more precisely, there exists $B>0$ depending on $F$ only such that for every $\ell \in \N$ we have
\[
\sup_{|\xi| \geq \kappa} \|\mathcal{L}_\xi^n \phi_\xi \|_{\infty} \leq B \, \| s \|_{\theta, 2, B\ell} \, n^{-\ell}.
\] 
Combining the bound above with \eqref{eq:from_p_K2} completes the proof.
\end{proof}

\begin{proposition}\label{prop:Krick_step1}
For any $x \in X$, define
\[
I_n(x) := \int_{-\infty}^\infty \overline{\psi_\xi(x)} \cdot (\mathcal{L}^n_\xi \phi_\xi)(x) \diff \xi.
\]
Then, 
\[
\int_{X \times \R} (r \circ F^n) \cdot \overline{s}  \, \diff \nu = \frac{1}{2\pi} \int_X \overline{I_n(x)} \diff \mu(x).
\]
\end{proposition}
\begin{proof}
For any fixed $x \in X$, the function 
\[
t \mapsto L^n_Fs(x,t) =  \sum_{\sigma^n y=x} e^{S_nu(y)} s(y, t - S_nf(y))
\]
is in $L^1(\R)$, since it is a finite sum of functions in $L^1(\R)$.

Thus, for any fixed $(x,\xi) \in X \times \R$, we have
\begin{equation}\label{eq:FT_Ln}
\begin{split}
\int_{-\infty}^{\infty} L^n_Fs(x,t) e^{-it\xi} \diff t &= \sum_{\sigma^n y = x} e^{S_nu(y)} \int_{-\infty}^{\infty} s(y,t-S_nf(y)) e^{-it\xi} \diff t = \sum_{\sigma^n y = x} e^{S_nu(y)- i \xi S_nf(y)} \widehat{s_y}(\xi) \\
&= (\mathcal{L}^n_\xi \phi_\xi)(x).
\end{split}
\end{equation}
By the Fubini-Tonelli Theorem,
\[
\int_{X \times \R} (r \circ F^n) \cdot \overline{s}  \, \diff \nu = \int_{X \times \R} r \cdot (L^n_F\overline{s})  \, \diff \nu = \int_X \int_{-\infty}^\infty r(x,t) \overline{L^n_Fs}(x,t) \diff t \diff \mu(x),
\]
where we used the fact that $u$ is real valued. 
By applying the inverse Fourier transform, we can write $r(x,t) = \frac{1}{2\pi} \int_{-\infty}^\infty \widehat{r_x}(\xi)e^{i t \xi} \diff \xi$. 
Using again the Fubini-Tonelli Theorem and \eqref{eq:FT_Ln}, we obtain
\begin{equation*}
\begin{split}
\int_{X \times \R} (r \circ F^n) \cdot \overline{s}  \, \diff \nu &=  \int_X \frac{1}{2\pi} \int_{-\infty}^\infty \int_{-\infty}^\infty \widehat{r_x}(\xi) \, \overline{L^n_Fs}(x,t) \, e^{i t \xi} \diff \xi \diff t \diff \mu(x) \\
&= \int_X  \frac{1}{2\pi} \int_{-\infty}^\infty \widehat{r_x}(\xi)\overline{ (\mathcal{L}^n_\xi \phi_\xi)(x)} \diff \xi \diff \mu(x),
\end{split}
\end{equation*}
which proves our claim.
\end{proof}

We are now ready to prove Theorem \ref{thm:main_1_sd}. 

\begin{proof}[Proof of Theorem \ref{thm:main_1_sd}]
Combining \Cref{prop:Krick_step1} and \Cref{lem:conseq}, we deduce that for every $k \in \N$ there exist constants $C_1, \dots, C_k>0$ such that 
\begin{equation*}
\begin{split}
\Bigg\lvert \int_{X \times \R} (r \circ F^n) \cdot \overline{s}  \, \diff \nu &- \frac{1}{2\pi} \int_X \Bigg( \sqrt{\frac{\pi}{\om}} \, \nu(s) \,  \Big( \int_\R \overline{r_x(\xi)} \diff \xi \Big) \,  n^{-\frac{1}{2}} + \sum_{\substack{3 \leq j\leq 2 k -1 \\ j \in 2\N +1}}  \overline{c_j(x)} n^{-\frac{j}{2}}  \Bigg) \diff \mu(x) \Bigg\rvert \\
& \leq C_k \|r\|_{\theta, Bk, Bk} \, \|s\|_{\theta, Bk, Bk} \, n^{-k},
\end{split}
\end{equation*}
where
\begin{equation*}
\begin{split}
| \overline{c_j(x)} | \leq C_j \|r\|_{\theta, j+2,0} \, \|s\|_{\theta, j+2,0}.
\end{split}
\end{equation*}
\sloppy In order to conclude, it is enough to define $c_j = c_j(r,s) := (2\pi)^{-1} \int_X \overline{c_j(x)} \diff \mu(x)$ and to note that $\int_X \int_\R r_x(t) \diff t \diff \mu(x) = \nu(r)$.
\end{proof}


\section{Approximating sequences}\label{Sec:6}

We introduce approximating sequences which we will use in the next section to reduce the problem to the case of a one-sided shift. For the remainder of the section, we assume that the function $f$ in the definition of the skew product $F$ depends only on the future coordinates.

\subsection{Approximating sequences and correlations}

Let us fix a constant $\beta \in (\sqrt{\theta},1)$ and let us consider an observable $s \in \mathscr{L}$. We do not assume that $s$ depends only on the future coordinates.

\begin{definition}
A sequence $\{s_m\}_{m \geq 0}$ in $\mathscr{L}$ is an approximating sequence for $s$ if, for all $m \geq 0$,
\begin{enumerate}
\item $\nu(s_m)=\nu(s)$,
\item $\| s_m\|_{\infty,p,q} \ll_p m^p \|s\|_{\theta, p,q}$ for all $p,q \geq 0$,
\item $|s_m|_{\beta,p,q} \ll_p m^p \beta^{-m} \|s\|_{\theta, p,q+2}$,
\item $\|s_m - s\circ F^m\|_{\infty,0,0} \ll \theta^m \|s\|_{\theta,0,0}$,
\item $s_m$ depends only on the future coordinates.
\end{enumerate}
\end{definition}

In the definition above, the implicit constants are independent of $s$, but they depend on $F$, $\theta$ and on $\beta$.
In \S\ref{sec:existence_app_seq}, we will prove the following result.

\begin{proposition}\label{prop:existence_app_seq}
Any $s \in \mathscr{L}$ has an approximating sequence.
\end{proposition}

For any $r,s \in \mathscr{L}$, we define
\[
\langle r,s \rangle = \int_{\Sigma \times \R} r \cdot \overline{s} \, \diff \nu.
\]
The following bound holds.

\begin{lemma}\label{lem:basic_app_seq}
For any $r,s \in \mathscr{L}$ and any $n \in \Z$ we have
\[
|\langle r\circ F^n, s\rangle| \ll \|r\|_{\infty,2,0} \cdot \|s\|_{\infty,0,0}.
\]
\end{lemma}
\begin{proof}
As in the proof of \Cref{lem:L1_norm_FT}, we have
\[
\| r\circ F^n\|_{L^1(\nu)} = \| r\|_{L^1(\nu)}\leq \sup_{x\in \Sigma} \|r_x\|_{L^1(\R)} \ll \sup_{x\in \Sigma} \|r_x\|_{2,0} = \|r\|_{\infty,2,0},
\]
therefore
\[
|\langle r\circ F^n, s\rangle| \leq \| r\circ F^n\|_{L^1(\nu)} \cdot \|s\|_\infty \ll \|r\|_{\infty,2,0} \cdot \|s\|_{\infty,0,0}.
\]
\end{proof}

We can deduce the following result.

\begin{proposition}\label{prop:sub_obs}
If $\{r_m\}_{m \geq 0}$ and $\{s_m\}_{m \geq 0}$ are approximating sequences for $r$ and $s$ respectively, then for every $n,m \geq 0$ we have
\[
|\langle r\circ F^n, s\rangle - \langle r_m\circ F^n, s_m\rangle| \ll m^2 \, \theta^m \, \|r\|_{\theta,2,0} \, \|s\|_{\theta,2,0}.
\]
\end{proposition}
\begin{proof}
We can write
\[
\begin{split}
\langle r_m\circ F^n, s_m\rangle &= \langle r_m\circ F^n, s_m - s \circ F^m\rangle +\langle r_m\circ F^n, s \circ F^m\rangle \\
&= \langle r_m\circ F^n, s_m - s \circ F^m\rangle +\langle r_m\circ F^n - r \circ F^{n+m}, s \circ F^m\rangle + \langle r \circ F^{n+m}, s \circ F^m\rangle.
\end{split}
\]
By measure invariance, the last term is equal to $\langle r \circ F^{n}, s \rangle$. Thus, by Lemma \ref{lem:basic_app_seq}, we deduce
\[
\begin{split}
|\langle r\circ F^n, s\rangle - \langle r_m\circ F^n, s_m\rangle| &\leq \|r_m\|_{\infty,2,0} \cdot \|s_m - s \circ F^m\|_{\infty,0,0} + \|s\|_{\infty,2,0} \cdot \|r_m - r \circ F^{m}\|_{\infty,0,0} \\
&\ll m^2 \, \theta^m \, \|r\|_{\theta,2,0} \, \|s\|_{\theta,2,0},
\end{split}
\]
which completes the proof.
\end{proof}

The rest of the section is devoted to the proof of Proposition \ref{prop:existence_app_seq}. We start by some preliminaries.

\subsection{Preliminaries}

We prove some estimates on Schwartz functions on $\R$ that we will need in the course of the proof of Proposition \ref{prop:existence_app_seq}.

\begin{lemma} \label{lemma:shiftv}
    For a Schwartz function $v \in \mathscr{S}$ and constant $a \in \R,$
    the shifted function $v_a(t) = v(a + t)$ satisfies
    \[
        \| v_a \| _{p,q} \ll_p |a|^p \| v \| _{0,q} + \| v \| _{p,q}.
    \] \end{lemma}   
\begin{proof}
    If $|t| \le 1,$ then $|v^{(q)}(t+a)| \le \| v \| _{0,q} \le \| v \| _{p,q}.$
    \ If $1 \le |t| \le 2|a|,$ then
    \[    
        |t|^p |v^{(q)}(t+a)|
        \le 2^p |a|^p |v^{(q)}(t+a)|
        \ll_p |a|^p \| v \| _{0,q}.
    \]
    If both $|t| \ge 1$ and $|t| \ge 2|a|,$ then $|t/(t+a)| \le 2$ and
    \[
        |t|^p |v^{(q)}(t+a)|
        \le 2^p |t+a|^p |v^{(q)}(t+a)|
        \ll_p \| v \| _{p,q}.
        \qedhere
    \] \end{proof}

\begin{lemma} \label{lemma:shiftvw}
    For Schwartz functions $v,w \in \mathscr{S}$ and constants $a,b \in \R,$
    the shifted functions
    $v_a(t) = v(a + t)$ and $w_b = w(b + t)$ satisfy
    \begin{align*}
        \| v_a - w_b \| _{p,q} \ \le \
            &A^p \| v - w \| _{0,q} \ + \ \| v - w \| _{p,q} \ +
        \\
            &\big( A^p \| v \| _{0,q+1} \ + \ \| v \| _{p,q+1} \big) \cdot |b-a|
    \end{align*}
    where $A = \max(|a|,|b|).$
\end{lemma}
\begin{proof}
    By the Mean Value Theorem,
    there is $\zeta \in [a,b]$ such that
    \[
        v^{(q)}(t+a) \ - \ v^{(q)}(t+b) \ = \ v^{(q+1)}(t + \zeta) \cdot (b - a).
    \]
    Then
    \[
        \| v_a - v_b \| _{p,q} \ \le
        \ \max_{\zeta \in [a,b]} \ \| v_\zeta \| _{p,q+1} \cdot |b-a|
    \]
    and \Cref{lemma:shiftv} implies that
    \[
        \| v_a - v_b \| _{p,q} \ \ll_p
        \ \big( A^p \| v \| _{0,q+1} \ + \ \| v \| _{p,q+1} \big) \cdot |b-a|.
    \]
    \Cref{lemma:shiftv} also implies that
    \[
        \| v_b - w_b \| _{p,q} \ \ll_p
        \ A^p \| v - w \| _{0,q} \ + \ \| v - w \| _{p,q}
    \]
    and the result follows from the triangle inequality.
\end{proof}
The inequalities above are a bit unwieldly, so we will use
the following simpler estimates.

\begin{corollary} \label{cor:simpleshift}
    With $v_a$ and $w_b$ as above and $A \ge \max(1,|a|,|b|),$
    \begin{align*}
        \| v_a \| _{p,q}
            &\ \ll_p A^p \ \| v \| _{p,q} \qandq
        \\
        \| v_a - w_b \| _{p,q}
            &\ \ll_p
                \ A^p \, \| v - w \| _{p,q} \ + \ A^p \| v \| _{p,q+1} \, |b-a|.
    \end{align*} \end{corollary}

We now look at the effect of composing $s \in \mathscr{L}$ with
the dynamics $F^m$ for $m \ge 0.$

\begin{lemma} \label{lemma:sFm}
    Let $s \in \mathscr{L}$. For $m \ge 0,$ and $\beta \in [\theta,1)$, the composition $s \circ F^m$ satisfies
    \[
        \| s \circ F^m \| _\ipq \ll_p m^p \| s \| _\ipq
        \qandq
        \| s \circ F^m \|_{\beta,p,q} \ll_{p} m^p \beta^{-m} \| s \| _\tpqq,
    \] 
where the implicit constant depends on $\beta$ as well.
\end{lemma}
\begin{proof}
    These inequalities both follow by combining
    \Cref{cor:simpleshift} and \Cref{lemma:birknorm} with $\beta$ in place of $\theta$.
    We prove the second inequality and leave the first for the reader.
    Consider points $x,y \in \Sigma.$
    Applying \Cref{cor:simpleshift} with
    $v = s(\sigma^m x, \cdot),$
    $w = s(\sigma^m y, \cdot),$
    $a = S_m f(x),$ and
    $b = S_m f(y),$
    and using
    \begin{align*}
        A
            &\ \le \| S_m f \| _\infty \ \ll \ m,
        \\
        |b - a|
            &\ \le \ |S_m f|_\beta \ d_{\beta}(x, y)
                \ \ll \ \beta^{-m} d_{\beta}(x, y),
        \\
        \| v \| _\pqq
            &\ \le \ \| s \| _\ipqq, \quad \text{and}
        \\
        \| v - w \| _\pq
            &\ \le \ |s|_\tpq \ d_{\beta}(\sigma^m x, \sigma^m y)
                \ \le \ |s|_\tpq \ \beta^{-m} d_{\beta}(x, y),
    \end{align*}
    we see that
    \begin{align*}
        \| v_a - w_b \| _\pq
            &\ \ll_{p} \ m^p |s|_\tpq \, \beta^{-m}  d_{\beta}(x,y)
                \ + \ m^p \| s \| _\ipqq \, \beta^{-m} d_{\beta}(x, y)
        \\
            &\ \ll_{p} m^p \beta^{-m}  \| s \| _\tpqq \, d_{\beta}(x,y).
                \qedhere
    \end{align*} \end{proof}

\subsection{Proof of Proposition \ref{prop:existence_app_seq}}\label{sec:existence_app_seq}

Recall that $\Sigma$ is a two-sided subshift of finite type where each
element $x \in \Sigma$ is given by a bi-infinite sequence
\[
    \ldots, x_{-3}, x_{-2}, x_{-1}, x_0, x_1, x_2, x_3, \ldots
\]
of letters in an alphabet $\mathscr{A}.$
For each letter $a \in \mathscr{A},$ fix a choice of a sequence
\[
    \ldots, a_{-3}, a_{-2}, a_{-1}, a_0, a_1, a_2, a_3, \ldots
\]
in $\Sig$ with $a_0 = a.$ Then for $x \in \Sig$ with $x_0 = a,$ we
define $\om_0(x) \in \Sig$ by
\[
    \ldots, a_{-3}, a_{-2}, a_{-1}, x_0, x_1, x_2, x_3, \ldots
\]
This defines a function $\om_0 : \Sig \to \Sig.$
For each $n > 0,$ further define $\om_n : \Sig \to \Sig$
by the conjugation $\om_n = \sig^n \circ \om_0 \circ \sig \invn.$
One can verify that these maps have the following properties:
\begin{enumerate}
    \item $\om_n(\om_n(x)) = \om_n(x)$
    \item
    $\dth(\om_n(x), x) < \theta^n$
    \item
    $\dth(\sig^m \om_n(x), \sig^m(x)) < \theta^{n+m}$
    \item
    $\dth(\om_n(x), \om_n(y)) \le \dth(x,y)$
\end{enumerate}    
for all $x,y \in \Sig$ and $m,n \ge 0.$

For $s \in \cL$ and $n \ge 0,$ define $\Om_n s : \Sig \ti \bbR \to \bbC$ by
\[
    (\Om_n s)(x,t) = s(\om_n(x), t).
\]
\begin{lemma} \label{lemma:OmsFm}
    For integers $n,m \ge 0,$ the function $\Om_n s \circ F^m$ satisfies
    \[
        \| \Om_n s \circ F^m \| _\ipq \ll_p m^p \| s \| _\ipq
        \qandq
        \| \Om_n s \circ F^m \| _\tpq \ll_p m^p \theta \invm \| s \| _\tpqq.
    \] \end{lemma}
\begin{proof}
    Note that
    \[
        (\Om_n s \circ F^m) (x,t)
        = \Om_n s \big(\sig^m x, t + S_m f(x) \big)
        = s \big(\om_n \sig^m x, t + S_m f(x) \big)
    \]
    and similarly
    \begin{math}
        (\Om_n s \circ F^m) (y,t) = s \big(\om_n \sig^m y, t + S_m f(y) \big) \end{math}
    where
    \[
        \dth(\om_n \sig^m x, \om_n \sig^m y) \le \dth(\sig^m x, \sig^m y).
    \]
    Adapting the proof of \Cref{lemma:sFm} with $v$ and $w$ replaced by
    $s(\om_n \sig^m x, \cdot)$ and $s(\om_n \sig^m y, \cdot)$ respectively,
    the result follows.
\end{proof}
%
%

\begin{lemma} \label{lemma:Omss}
    For integers $n,m \ge 0,$ the function
    $(\Om_n s - s) \circ F^m = \Om_n s \circ F^m - s \circ F^m$
    satisfies
    \[
        \| (\Om_n s - s) \circ F^m \| _\ipq \ \ll_p \ m^p \theta^n \, |s|_\tpq.
    \] \end{lemma}
\begin{proof}
    Since
    $\dth(\omega_n \sig^m x, \sig^m x) < \theta^n,$
    the result follows from \Cref{cor:simpleshift} with
    $v = s(\omega_n \sig^m x, \cdot),$
    $w = s(\sig^m x, \cdot),$
    and
    $a = b = S_m f(x).$
\end{proof}

For results later in this section, we will need the following estimates.

\begin{lemma} \label{lemma:sumparts}
    For positive integers $m$ and $p,$
    \[
        \sum_{n=m}^\infty n^p \theta^n \ll_p m^p \theta^m
        \qandq
        \sum_{n=0}^m n^p \theta \invn \ll_p m^p \theta \invm.
    \] \end{lemma}
\begin{proof}
    We prove the first inequality.
    The proof of the second inequality is similar.
    Using summation by parts,
    \[
        (\theta - 1) \sum_{n=m}^\infty n^p \theta^n
        = m^p \theta^m + \sum_{n=m+1}^\infty \big( n^p - (n-1)^p \big) \theta^m.
    \]
    Since $n^p - (n-1)^p$ is a polynomial in $n$ of degree $p - 1,$
    we can assume using induction that
    \begin{math}
        \sum_{n=m+1}^\infty \big(n^p - (n-1)^p \big) \theta^m
        \ll_p m^{p-1} \theta^m.
    \end{math} \end{proof}
Assume $s \in \mathscr{L}$ is given and for $m \ge 0,$
define a function $v_m : \Sig \ti \bbR \to \bbC$
by
\[
    v_m = \sum_{n=m}^\infty (\Om_n s - s) \circ F^n.
\]
We can see from \Cref{lemma:Omss} that the series converges uniformly.
In fact, we have the following.

\begin{lemma} \label{lemma:vmipq}
    For any $m,$
    we have $v_m \colon \Sigma \to \mathscr{S}$ and 
    $ \| v_m \| _\ipq \ll_p m^p \theta^m |s|_\tpq.$
\end{lemma}
\begin{proof}
    This follows from \Cref{lemma:Omss} and \Cref{lemma:sumparts}.
\end{proof}

Define for each $m$ a function
\begin{equation}\label{eq:defin_sm}
s_m := s \circ F^m + v_m - v_m \circ F.
\end{equation}
From the definition of $v_m$, we immediately get the following.
\begin{lemma}\label{lem:reduction_4}
The function $s_m$
only depends on the future coordinates. 
\end{lemma}
\begin{proof}
It is clear that $\Omega_ns \circ F^n$ only depends on the future coordinates of $x$.
By the definition of $v_m$, the function 
\[
s_m =  s \circ F^m + v_m - v_m \circ F = \sum_{n=m}^\infty \Omega_ns\circ F^n(x,t) - \Omega_ns\circ F^{n+1}(x,t)
\]
only depends on the future coordinates. 
\end{proof}

Estimating the regularity of $v_m$ is more complicated.
In fact, the function is only H\"older with respect to $\dth$.

\begin{lemma} \label{lemma:vmholder}
    If $N > 0$ and $x,y \in \Sig$ are such that $\dth(x,y) < \theta^{2N},$
    then
    \[
        \| v_m(x, \cdot) - v_m(y, \cdot) \| _\pq
        \ll_p 
        N^p \theta^N \| s \| _\tpqq.
    \] \end{lemma}
\begin{proof}
    For simplicity, we assume $m < N.$
    For all $n \ge m,$ define
    \[
        E_n \ = \  \bigg \|
        \Om_n s \circ F^n(x,\cdot)
        \ - \ \Om_n s \circ F^n(y,\cdot)
        \ - \ s \circ F^n(x,\cdot)
        \ + \ s \circ F^n(y,\cdot)
        \bigg \| _\pq
    \]
    so that 
    $ \| v_m(x, \cdot) - v_m(y, \cdot) \| _\pq \le \sum_{n=m}^\infty E_n.$
    If $m \le n < N,$ \Cref{lemma:sFm} and \Cref{lemma:OmsFm} imply that
    \[
        E_n
        \ \ll_p \ n^p \theta \invn \| s \| _\tpqq \dth(x,y)
        \ < \ n^p \theta^{2N-n} \| s \| _\tpqq 
    \]
    If $n \ge N,$ \Cref{lemma:Omss} implies that
    \[
        E_n \ \ll_p \ n^p \theta^n \, |s|_\tpq.
    \]
    Then using \Cref{lemma:sumparts},
    \begin{align*}
        \sum_{n=m}^\infty E_n
            &\ \ll_p \ \sum_{n=m}^{N-1} n^p \theta^{2N-n} \| s \| _\tpqq
                \ + \ \sum_{n=N}^\infty n^p \theta^n |s|_\tpq
        \\
            &\ \ll_p \ N^p \theta^N \| s \| _\tpqq \ + N^p \theta^N |s|_\tpq.
                \qedhere
    \end{align*} \end{proof}
We now consider $\Sig$ equipped with a metric $d_\beta$ coming from
a different constant $0 < \beta < 1.$

\begin{lemma} \label{lemma:betalip}
    If $\sqrt{\theta} < \beta < 1,$ then $v_m$ is Lipschitz with respect to
    $d_\beta$ and
    \[
        | v_m | _{\beta,p,q} \ll_p \| s \| _\tpqq.
    \] \end{lemma}
Here, as before, the constant factor associated to the $\ll_p$ notation
depends on all of $f$, $p$, $\theta$, and $\beta$,
but it is independent of $q$, $m$, and $s$.

\begin{proof}
    Since $\lim_{N \to \infty} N^p \theta^N \beta^{-2N} = 0$ for any fixed $p,$
    there is a constant $D_p$ such that $N^p \theta^N \le D_p \beta^{2N}$
    for all $N.$
    Note that
    $d_\beta(x,y) < \beta^{2N}$
    is equivalent to
    $d_\theta(x,y) < \theta^{2N}$
    and by \Cref{lemma:vmholder} this implies that
    \[
        \| v_m(x, \cdot) - v_m(y, \cdot) \| _\pq \ll_p \beta^{2N} \| s \| _\tpqq.
    \]
    From this, the result follows.
\end{proof}

\begin{corollary} \label{cor:betalipF}
    If $\sqrt{\theta} < \beta < 1,$ then $v_m \circ F$ is Lipschitz with respect to
    $d_\beta$ and
    \[
        | v_m \circ F | _{\beta,p,q} \ll_p (1+m^p \beta^m)\| s \| _{\theta,p,q+2}.
    \] \end{corollary}
\begin{proof}
    As in \Cref{lemma:sFm}, we have
    \[
       | v_m \circ F | _{\beta,p,q} \ll_p \| v_m\|_{\beta,p,q+1},
    \]
  so that the conclusion follows from \Cref{lemma:betalip} and \Cref{lemma:vmipq}.
\end{proof}

\begin{proof}[{Proof of Proposition \ref{prop:existence_app_seq}}]
We define $s_m$ as in \eqref{eq:defin_sm}. It is clear that $\nu(s_m) = \nu(s \circ F^m) = \nu(s)$.
By Lemma \ref{lem:reduction_4}, $s_m$ only depends on the future coordinates.

Combining Lemma \ref{lemma:sFm} and Lemma \ref{lemma:vmipq}, we get $\|s_m\|_{\infty, p,q} \ll_p m^p \|s\|_{\theta, p,q}$ for all $p,q \geq 0$.
By the same token, we also get 
\[
\|s_m - s\circ F^m\|_{\infty,0,0} \leq 2 \|v_m\|_{\infty,0,0} \ll \theta^m \|s\|_{\theta, 0,0}.
\]
Finally, from Lemma \ref{lemma:sFm}, Lemma \ref{lemma:betalip}, and Corollary \ref{cor:betalipF}, we obtain $|s_m|_{\beta, p,q} \ll_p m^p \beta^{-m} \|s\|_{\theta, p,q+2}$. The proof is complete.
\end{proof}

\section{Reduction from two-sided to one-sided skew products}\label{sec:redu}

In this section we are going to describe how, starting from a skew product over a two-sided subshift, we can reduce ourselves to study the correlations for a skew product over a one-sided subshift instead. 
The approach is similar to our previous work in \cite[Section 10]{GHR}, but more precise estimates are needed here. 

\begin{theorem}\label{thm:reduction}
For any $\theta^{1/4} < \beta <1$, there exists $\widetilde{f} \in \mathscr{F}^{+}_{\beta}$ such that the following statement  holds.
Let $r,s \in \mathscr{L}_\theta$. There exist approximating sequences $\{r_m\}_m, \{s_m\}_m$ such that for every $m \geq 0$ 
\[
    \left\lvert \langle s \circ F^n, r \rangle - \int_{X \times \R}(s_m \circ \widetilde{F}^n) \cdot \overline{r_m} \diff \nu^{+} \right\rvert \ll m^2 \beta^m \|s\|_{\theta, 2,1} \,\|r\|_{\theta, 2,1},
\]
where $\widetilde{F}(x,t) = (\sigma x, t + \widetilde{f}(x))$, $\nu^{+} = \mu^{+} \times \Leb$, and $\mu^{+}$ is the push-forward of $\mu$ onto $X$ under the canonical projection.
\end{theorem}

In other words, Theorem \ref{thm:reduction} says that we can approximate the correlations for the skew product over a two-sided shift with correlations for another skew product over a one-sided shift up to an arbitrarily small error. The price to pay is that we have a loss in the regularity of the functions, since we have to consider a larger $\beta$, and moreover both $\infpq{s_m}$ and $\ltttpq{s_m} $ grow as the error gets small. 

Let us prove Theorem \ref{thm:reduction}.
The following result is classical, see, e.g., \cite[Proposition 1.2]{PaPo}; in our case, it follows from the results we proved in \S\ref{Sec:6}.

\begin{lemma}\label{lem:reduction_PP}
For any $f \in \mathscr{F}_{\theta}$ there exist $h \in \mathscr{F}_{\sqrt{\theta}}$ and $\widetilde{f} \in \mathscr{F}_{\sqrt{\theta}}^{+}$ such that $f = \widetilde{f} + h - h \circ \sigma$.
\end{lemma}
\begin{proof}
    Consider the (non skew) product $F_0$ of the form $F_0 (x,t) = (\sigma x, t)$.
Applying \Cref{lem:reduction_4} with $F_0$ in place of $F$ and with $s(x,t) = f(x)$ and $m=0$, we deduce that $f +h - h\circ \sigma$ depends only on the future coordinates, where $h(x) = v_0(x,t)$ is independent of $t$. 
The regularity of $h$  is then given by adapting the proof of \Cref{lemma:betalip}.
\end{proof}

Let us denote by $\widetilde{F}$ the skew product $\widetilde{F}(x,t) = (\sigma x, t+ \widetilde{f}(x))$. We now show that, up to replacing the observables and considering $\sqrt{\theta}$ instead of $\theta$, we can exchange $F$ with $\widetilde{F}$.

\begin{lemma}\label{lem:reduction_3}
For any $s \in \mathscr{L}$ there exists $\widetilde{s} \in \mathscr{L}$ with $\nu(\widetilde{s}) = \nu(s)$ and 
\[
\begin{split}
&\infpq{\widetilde{s}} \ll_p  \infpq{s},\\
&\lttpq{\widetilde{s}} \ll_p \|s\|_{\theta, p,q+1},
\end{split}
\]
such that the following holds. Let  $s, r \in \mathscr{L}$; then,
\[
\langle r \circ F^n, s \rangle = \langle \widetilde{r} \circ \widetilde{F}^n, \widetilde{s} \rangle,
\]
for all $n \in \N$.
\end{lemma}
\begin{proof}
Let  $s \in \mathscr{L}$. We define $\widetilde{s}(x,t) = s(x, t-h(x))$, where $h$ is given by Lemma \ref{lem:reduction_PP}. Clearly, for any $x \in \Sigma$, the function $\widetilde{s}(x) \in \mathscr{S}$ and also $\nu(\widetilde{s}) = \nu(s)$. 

Let us fix $x \in \Sigma$, and let $w \in \mathscr{S}$ be the function $w(t) = \widetilde{s}(x,t)$. By \Cref{lemma:shiftv}, since $|h(x)| \ll 1$, we have
\[
\|w\|_\pq  \ll_p  \, \|s(x)\|_{0,q} + \|s(x)\|_{p,q} \leq 2 \infpq{s}.
\]
This proves the bound
\[
\infpq{\widetilde{s}} \ll_p \infpq{s}.
\]

Let now $x,y \in \Sigma$ and $w_1(t) =\widetilde{s}(x,t), w_2(t) = \widetilde{s}(y,t) $. Then, by \Cref{cor:simpleshift}, 
\[
\begin{split}
\|w_1-w_2\|_\pq \ll_p & \|s(x)-s(y)\|_{p,q} + \|s(x)\|_{p,q+1}\, |h(x)-h(y)| \\
\leq & \ltpq{s} \, d_\theta(x,y) + \infpqq{s} \, d_{\sqrt{\theta}}(x,y),
\end{split}
\]
where we used the fact that $h \in \mathscr{F}_{\sqrt{\theta}}$.
This proves the bound
\[
    \lttpq{\widetilde{s}} \ll_p  \ltpq{s} + \infpqq{s} \leq \|s\|_{\theta, p,q+1}.
\]
Finally, a change of variable gives us
\[
\begin{split}
\langle r \circ F^n, s \rangle &= \int_{\Sigma \times \R} r(\sigma^n x,t + S_nf(x)) \cdot \overline{s(x,t)} \diff \nu(x,t) \\
& \qquad =\int_{\Sigma \times \R} r(\sigma^n x,t + S_n\widetilde{f}(x) + h(x) - h(\sigma^nx)) \cdot \overline{s(x,t)} \diff \nu(x,t) \\
& \qquad =\int_{\Sigma \times \R} \widetilde{r}(\sigma^n x,t + S_n\widetilde{f}(x) ) \cdot \overline{\widetilde{s}(x,t)} \diff \nu(x,t)  =  \langle \widetilde{r} \circ \widetilde{F}^n, \widetilde{s} \rangle,
\end{split}
\]
hence the proof is complete.
\end{proof}
The previous lemma shows that we can assume that the function $f$ in the definition of the skew product only depends on the future coordinates.

\begin{proof}[Proof of Theorem \ref{thm:reduction}]
We first apply Lemma \ref{lem:reduction_3} to replace the skew product $F$ with $\widetilde{F}$, so that 
\[
\langle r \circ F^n, s \rangle = \langle \widetilde{r} \circ \widetilde{F}^n, \widetilde{s} \rangle.
\]
In doing so, we also replaced $\theta$ with $\beta = \sqrt{\theta}$.
Then, we replace the observables $\widetilde{s}$ and $\widetilde{r}$ with the new $\widetilde{s}_m$ and $\widetilde{r}_m$ given by Proposition \ref{prop:existence_app_seq}. By Proposition \ref{prop:sub_obs}, the error in the integral above can be bounded by
\[
\begin{split}
&\left\lvert \langle r \circ F^n, s \rangle -  \langle \widetilde{r}_m \circ \widetilde{F}^n, \widetilde{s}_m \rangle\right\rvert =\left\lvert \langle \widetilde{r} \circ \widetilde{F}^n, \widetilde{s} \rangle -  \langle \widetilde{r}_m \circ \widetilde{F}^n, \widetilde{s}_m \rangle\right\rvert \ll m^2 \theta^{m/2} \|\widetilde{r}\|_{\sqrt{\theta},2,0} \cdot \|\widetilde{s}\|_{\sqrt{\theta},2,0}.
\end{split}
\]
The estimates on the norms have the same form by Lemma \ref{lem:reduction_3}, up to fixing any $\beta \in (\theta^{1/4},1)$ and up to replacing $q+2$ with $q+3$ in the bound for the Lipschitz semi-norm. 
The proof of the theorem is therefore complete.
\end{proof}

As seen in the above proof, the sequences $\{r_n\}$ and $\{s_n\}$
in the statement of \Cref{thm:reduction} 
are actually approximating sequences for $\tilde r$ and $\tilde s$
instead of the original $r$ and $s$.
However, by a slight abuse of notation, we will still call
$\{r_n\}$ and $\{s_n\}$ approximating sequences for $r$ and $s$
if they satisfy the conclusion of the theorem.

\section{Proof of Theorem \ref{cor:main_results} }\label{Sec:7}

In this section we are going to prove \Cref{cor:main_results}.
The strategy is to reduce the problem to the one-sided case using \Cref{thm:reduction}. 
We now fix $\beta \in (\theta^{1/4},1)$ keeping in mind that the implicit constants in the following statements depend on this choice.

Let $r,s \in \mathscr{L}$ and $m \in \Z_{\geq 0}$; let us also fix approximating sequences $\{r_m\}$ and $\{s_m\}$ as given by \Cref{thm:reduction}.
We define
\[
\psi_\xi(x) := \widehat{(r_m)_x}(\xi), \qquad  \phi_\xi(x) := \widehat{(s_m)_x}(\xi), \qquad \text{and} \qquad \Phi_\xi(x) := \mathcal{L}^m_\xi \phi_\xi(x).
\]
which, by  \Cref{prop:FT_smooth_curves}, are smooth curves in $\mathscr{F}_{\beta}^+$.

\begin{lemma}\label{lemma:phiellxi}
For any $m \geq 1$, we have
\[
    \sup_{x\in X} \int_\R |\psi_\xi(x)| \diff \xi \ll m^2 \, \|r\|_{\infty, 2, 1},
    \qquad
    \infnorm(\psi,k) \ll_k m^{k+2} \, \|r\|_{\theta, k+2,1},
\]
and
\[
M(\Phi,k) \ll_k m^{2k+4} \, \|s\|_{\theta, k+2,2}.
\]
\end{lemma}
\begin{proof}
The estimates follow immediately from   \Cref{prop:FT_smooth_curves} and the properties of approximating sequences. 
\end{proof}

Given $n \in \Z_{\geq 0}$, let us define 
\[
m = m(n) := \left\lfloor 2 \frac{\log n}{-\log \beta} \right\rfloor +1 \ll \log n.
\]
From \Cref{prop:rapid_decay} applied with $\ell = 2$, we deduce
\begin{equation}\label{eq:from_prop_10}
    \sup_{|\xi| \geq \kappa} \|\mathcal{L}_\xi^n \phi_\xi \|_{\infty} \ll n^{-2} \|s\|_{\theta, 2, B},
\end{equation}
for some constant $B$ depending on $F$ only.

By \Cref{thm:reduction}, 
\[
\left\lvert \langle r \circ F^n, s \rangle - \int_{X \times \R}(r_m \circ F^n) \cdot \overline{s_m} \diff \nu \right\rvert \ll
m^2 \, \beta^m \, \|r\|_{\theta, 2,1} \,\|s\|_{\theta, 2,1},
\] 
where, by a little abuse of notation, we wrote $F$ to denote also the skew product over the one-sided shift and $\nu$ the push-forward measure under the canonical projection $\Sigma \to X$.
By our choice of $m$, we have that $m^2 \beta^m \ll (\log n)^2\, n^{-2}$, so that we can limit ourselves to consider the second integral in the left hand side above.
By \Cref{prop:Krick_step1}, this latter integral can be rewritten as
\[
\begin{split}
&\int_{X \times \R}(r_m \circ F^n) \cdot \overline{s_m} \diff \nu  = \frac{1}{2\pi} \int_X \overline{I_n(x)} \diff \mu(x), \qquad \text{where} \\
 &I_n(x) := \int_{-\infty}^\infty \overline{\psi_\xi(x)} \cdot (\mathcal{L}^{n}_\xi \phi_\xi)(x) \diff \xi = \int_{-\infty}^\infty \overline{\psi_\xi(x)} \cdot (\mathcal{L}^{n-m}_\xi \Phi_\xi)(x) \diff \xi.
\end{split}
\]

By \Cref{prop:Krick_step2}, \Cref{lemma:phiellxi}, and equation \eqref{eq:from_prop_10}, there exist $c_1(x), c_3(x) \in \C$ such that
\begin{equation}\label{eq:finall}
\begin{split}
    \left\lvert c_1(x) (n-m)^{-1/2} - I_n(x) \right\rvert \ll
    &|c_3(x)| \, (n-m)^{-3/2}
    + m^2 n^{-2} \, \|r\|_{\infty, 2, 1} \, \|s\|_{\theta, 2, B} \, \\
    &+ m^{12} n^{-2} \, \|r\|_{\theta, k+2,1} \, \|s\|_{\theta, k+2,2},
\end{split}
\end{equation}
where, using the fact that $\mu(\Phi_0) = \mu(\mathscr{L}_0^m\phi_0) = \mu(\phi_0) = \nu(s_m) = \nu(s)$, 
\[
c_1(x) = \sqrt{\frac{\pi}{\omega}} \, \mu(\Phi_0) \overline{\psi_0(x)}=  \sqrt{\frac{\pi}{\omega}} \, \nu(s) \overline{\psi_0(x)}.
\]

Note that $(n-m)^{-\frac{j}{2}}-n^{-\frac{j}{2}} \ll n^{-\frac{j}{2} -1} \log n$ for $j=1,2$. 
Then, from \eqref{eq:finall} and \Cref{lemma:phiellxi}, we deduce
\[
\left\lvert n^{1/2} I_n(x) -  \sqrt{\frac{\pi}{\omega}} \, \nu(s) \overline{\psi_0(x)} \right\rvert \ll \infnorm(\psi,2) M(\Phi,2) n^{-1} \ll
(\log n)^{10} n^{-1}
\|r\|_{\infty, B, B} \, \|s\|_{\theta, B, B}.
\]
We conclude that 
\[
\left\lvert n^{1/2} \langle r \circ F^n, s \rangle - \frac{1}{2\sqrt{\pi\, \omega}} \nu(\overline{s}) \int_X \psi_0(x) \diff \mu(x) \right\rvert  \ll
(\log n)^{10} n^{-1}
\|r\|_{\infty, B, B} \, \|s\|_{\theta, B, B},
\] 
which completes the proof of \Cref{cor:main_results}.

\end{document}